\documentclass[12pt,reqno]{amsart}
\usepackage{amsmath,amsthm,amsfonts}

\newcommand{\nexteq}{\displaybreak[0]\\ &=}
\newcommand{\nnexteq}{\notag\displaybreak[0]\\ &=}

\newtheorem{thm}{Theorem}
\newtheorem{prop}{Proposition}
\newtheorem{lem}{Lemma}
\newtheorem{cor}{Corollary}

\theoremstyle{definition}
\newtheorem{rem}{Remark}
\newtheorem{exa}{Example}

\DeclareMathOperator{\tr}{tr}
\DeclareMathOperator{\diag}{diag}
\newcommand{\C}{\mathbb{C}}
\newcommand{\R}{\mathbb{R}}
\newcommand{\Q}{\mathbb{Q}}
\newcommand{\Z}{\mathbb{Z}}

\begin{document}
\title[Complex Hadamard Matrices]{Complex Hadamard Matrices \\
Contained in a Bose--Mesner Algebra}
\author{Takuya Ikuta}
\address{Kobe Gakuin University, Kobe, 650-8586, Japan}
\email{ikuta@law.kobegakuin.ac.jp}
\author{Akihiro Munemasa}
\address{Tohoku University, Sendai, 980-8579, Japan}
\email{munemasa@math.is.tohoku.ac.jp}
\thanks{The work of T.I. was supported by JSPS KAKENHI grant number 
25400215, and that of A.M. was supported by JSPS KAKENHI grant number
26400003.}

\dedicatory{Dedicated to Chris Godsil on the occasion of his 65th birthday}
\date{April 21, 2015}
\keywords{association scheme, complex Hadamard matrix, type-II matrix}
\subjclass[2010]{05E30,05B34}
\begin{abstract}
A complex Hadamard matrix is a square matrix $H$ with complex 
entries of absolute value $1$ satisfying $HH^\ast= nI$, 
where $\ast$ stands for the Hermitian transpose and $I$ is the 
identity matrix of order $n$. 
In this paper, we first determine
the image of a certain rational map
from the $d$-dimensional complex projective space to $\mathbb{C}^{d(d+1)/2}$.
Applying this result with $d=3$, we 
give constructions of complex Hadamard matrices,
and more generally, type-II matrices,
in the Bose--Mesner algebra of a certain 3-class symmetric association scheme. 
In particular, we recover the complex Hadamard matrices of order $15$ found by 
Ada Chan. 
We compute the Haagerup sets to show inequivalence of resulting type-II matrices,
and determine the Nomura algebras to show that the resulting matrices are
not decomposable into generalized tensor products.
\end{abstract}
\maketitle

\section{Introduction}
A complex Hadamard matrix is a square matrix $H$ with complex 
entries of absolute value $1$ satisfying
$HH^\ast= nI$, 
where $\ast$ stands for the Hermitian transpose and $I$ is the identity matrix of order $n$.
They are the natural generalization of real Hadamard matrices. 
Complex Hadamard matrices appear frequently in various branches of mathematics and 
quantum physics.

A type-II matrix, or an inverse orthogonal matrix,
is a square matrix $W$ with nonzero complex entries satisfying
$W{W^{(-)}}^\top=nI$, where $W^{(-)}$ denotes the entrywise inverse of $W$. 
Obviously, a complex Hadamard matrix is a
type-II matrix.
Two type-II matrices $W_1$ and $W_2$ are said to be equivalent if
there exist diagonal matrices $D, D'$ with nonzero complex diagonal entries,
and permutation matrices $T, T'$, such that
$DW_1D'=TW_2T'$ holds.

Complete classifications of complex Hadamard matrices,
and of type-II matrices are 
only available up to order $n = 5$
(see \cite{Craigen, Nomura, Haagerup}).
Although it is shown by Craigen \cite{Craigen} that there are 
uncountably many equivalence
classes of complex Hadamard matrices of order $n$ whenever $n$ is a 
composite number, some type-II matrices are more closely
related to combinatorial objects than the others. 
Szollosi \cite{Szollosi10} used
design theoretical methods to construct complex Hadamard matrices.
Strongly regular graphs were used to construct type-II matrices
in \cite{ChanGodsil,ChanHosoya}. See \cite{Sankey} for a generalization.
In this paper, we construct type-II matrices and complex Hadamard matrices
in the Bose--Mesner algebra of a certain 3-class symmetric association scheme. 
In particular, we recover the complex Hadamard matrix of order $15$ found in 
\cite{Chan}.

The method of finding complex Hadamard matrices in
the Bose--Mesner algebra of a symmetric association scheme
generalizes the classical work of Goethals and Seidel \cite{GoethalsSeidel}.
Assuming that the association scheme is symmetric, the resulting
complex Hadamard matrices are symmetric. It turns out that this assumption
enables us to consider only the real parts of the entries of a complex
Hadamard matrix, since the orthogonality can be expressed in terms of
the real parts. Extending this reduction to type-II matrices, we are
led to consider a rational map whose inverse is explicitly given in
Section~\ref{sec:2}. In Section~\ref{sec:3}, we explain why only real parts
come into play when we construct complex Hadamard matrices in the
Bose--Mesner algebra of a symmetric association scheme.
In Section~\ref{sec:4}, we consider a particular family of 3-class association
schemes. This family was found after extensive computer experimentation on the list
of 3-class association schemes up to $100$ vertices given in \cite{vanDam}.
Surprisingly, most other association schemes up to $100$ vertices,
with the exceptions of amorphic or pseudocyclic schemes, do not admit
a complex Hadamard matrix in their Bose--Mesner algebras.
In Section~\ref{sec:5}, we compute the Haagerup set
to show inequivalence of 
type-II matrices 
constructed in Section~\ref{sec:4}.
In Section~\ref{sec:6}, we show that the Nomura algebra of
each of the type-II matrices
constructed in Section~\ref{sec:4} has dimension $2$.
This implies that our matrices are not equivalent to
generalized tensor products defined in \cite{HS}.

All the computer calculations in this paper were performed by
Magma \cite{magma}.

\section{The image of a rational map}\label{sec:2}
We define a polynomial in three indeterminates $X,Y,Z$ as follows:
\[
g(X,Y,Z)=X^2+Y^2+Z^2-XYZ-4.
\]

\begin{lem}\label{lem:g}
\[
g\left(\frac{X}{Y}+\frac{Y}{X},\frac{X}{Z}+\frac{Z}{X},\frac{Z}{Y}+\frac{Y}{Z}\right)=0.
\]
\end{lem}
\begin{proof}
Straightforward.
\end{proof}

\begin{lem}\label{lem:1}
In the rational function field with four indeterminates $X,Y,Z$ and $z$,
the following identity holds:
\begin{align}
ww'&=1+\frac{z^2g+(2zX-zYZ+f)f}{z(zZ-Y)(zY-Z)},
\label{c3}
\end{align}
where
\begin{align*}
f&=z^2-zX+1,\\
g&=g(X,Y,Z),\\
w&=\frac{z^2-1}{zZ-Y},\\
w'&=\frac{z^{-2}-1}{z^{-1}Z-Y}.
\end{align*}
\end{lem}
\begin{proof}
Straightforward.
\end{proof}

We define a polynomial in six indeterminates 
$X_{0,1},X_{0,2},X_{0,3},X_{1,2},X_{1,3},X_{2,3}$ as follows:
\[
h(X_{0,1},X_{0,2},X_{0,3},X_{1,2},X_{1,3},X_{2,3})
=\det\begin{bmatrix}
2&X_{0,1}&X_{0,2}\\
X_{0,1}&2&X_{1,2}\\
X_{0,3}&X_{1,3}&X_{2,3}
\end{bmatrix}.
\]

\begin{lem}\label{lem:h}
In the rational function field with four indeterminates $X_0,X_1,X_2,X_3$,
set
\[
x_{i,j}=\frac{X_i}{X_j}+\frac{X_j}{X_i}\quad(0\leq i<j\leq3).
\]
Then 
$h(x_{0,1},x_{0,2},x_{0,3},x_{1,2},x_{1,3},x_{2,3})=0$.
\end{lem}
\begin{proof}
Straightforward.
\end{proof}

For a finite set $N$ and a positive integer $k$, 
we denote by $\binom{N}{k}$ the collection of all $k$-element subsets
of $N$.

\begin{lem}\label{lem:3}
Let $N=\{0,1,\dots,d\}$, 
$N_3=\binom{N}{3}$ and
$N_4=\binom{N}{4}$.
Let $a_{i,j}$ $(0\leq i,j\leq d,\;i\neq j)$ be
complex numbers satisfying
\begin{align}
a_{i,j}&=a_{j,i}\quad(0\leq i<j\leq d),\label{aij}\\
g(a_{i,j},a_{j,k},a_{i,k})&=0\quad(\{i,j,k\}\in N_3),\label{N1}\\
h(a_{i,j},a_{i,k},a_{i,\ell},a_{j,k},a_{j,\ell},a_{k,\ell})&=0
\quad(\{i,j,k,\ell\}\in N_4).\label{N4}
\end{align}
Assume 
\begin{equation}\label{a1}
a_{i_0,i_1}\neq\pm2\quad\text{for some $i_0,i_1$ with
$0\leq i_0<i_1\leq d$.}
\end{equation}
Let $w_{i_0}$, $w_{i_1}$ be nonzero complex numbers satisfying
\begin{equation}\label{az1d}
\frac{w_{i_0}}{w_{i_1}}+\frac{w_{i_1}}{w_{i_0}}=a_{{i_0,i_1}}.
\end{equation}
Then for
complex numbers $w_i$ $(0\leq i\leq d,\;i\neq i_0,i_1)$,
the following are equivalent:
\begin{enumerate}
\item for all $i,j$ with $0\leq i,j\leq d$ and $i\neq j$,
\begin{equation}\label{az2}
\frac{w_j}{w_i}+\frac{w_i}{w_j}=a_{i,j}
\end{equation}
\item
for all $i,j$ with $0\leq i\leq d$, $i\neq i_0,i_1$,
\begin{equation}\label{azid}
w_i=
\frac{w_{i_1}^2-w_{i_0}^2}{a_{{i_1},i}w_{i_1}-a_{{i_0},i}w_{i_0}}.
\end{equation}
\end{enumerate}
Moreover, if one of the two equivalent conditions
{\rm(i), (ii)} is satisfied,
$a_{i,j}$ $(0\leq i<j\leq d)$ are all real and
\begin{equation}\label{-22}
-2<a_{i_0,i_1}<2, 
\end{equation}
then $|w_i|=|w_j|$ for $0\leq i<j\leq d$.
\end{lem}
\begin{proof}
Without loss of generality, we may assume $i_0=0$ and $i_1=1$. Thus
(\ref{az1d}) reads
\begin{equation}\label{az1d01}
\frac{w_0}{w_1}+\frac{w_1}{w_0}=a_{0,1}.
\end{equation}
Then by (\ref{a1}) and (\ref{N1}) with $\{i,j,k\}=\{0,1,i\}$,
we find
\begin{equation}\label{a1x}
a_{0,i}^2+a_{1,i}^2-a_{0,1}a_{0,i}a_{1,i}\neq0.
\end{equation}
First we need to check that the denominator of (\ref{azid}) is nonzero.
Suppose $a_{1,i}w_1-a_{0,i}w_0=0$. If, moreover $a_{1,i}=0$, then
we have $a_{0,i}=0$, but this contradicts (\ref{a1x}).
Thus we have $a_{1,i}\neq0$, and 
\[
\frac{w_1}{w_0}=\frac{a_{0,i}}{a_{1,i}}.
\]
Then by (\ref{az1d01}), we have 
\[
\frac{a_{1,i}}{a_{0,i}}+\frac{a_{0,i}}{a_{1,i}}=a_{0,1},
\]
but again this contradicts (\ref{a1x}).
Therefore, $w_i$ is well-defined.
Moreover, we claim $w_i\neq0$. Indeed, $w_i=0$ would imply
$w_1^2=w_0^2$.
Then by (\ref{az1d01}), we have $a_{0,1}=\pm2$,
which contradicts (\ref{a1}).

Clearly, (\ref{az1d01}) implies (\ref{az2}) with $(i,j)=(0,1)$.
Suppose $j\geq2$. 
Observe that $x=w_j=(w_1^2-w_0^2)/(a_{1,j}w_1-a_{0,j}w_0)$ 
given in (\ref{azid}) is a unique common root of the
equations
\begin{align*}
x^2-w_0 a_{0,j}x+w_0^2&=0,\\
x^2-w_1 a_{1,j}x+w_1^2&=0.
\end{align*}
This can be seen by $g(a_{0,1},a_{0,j},a_{1,j})=0$.
Thus (\ref{az2}) holds if $i=0$ or $1$.
Now assume $\{i,j\}\cap\{0,1\}=\emptyset$.
By (\ref{N4}), we have
\begin{align*}
0&=w_0^2w_1^2w_iw_jh(a_{0,1},a_{0,i},a_{0,j},a_{1,i},a_{1,j},a_{i,j})
\nexteq
\det\begin{bmatrix}
2w_0^2&w_0w_1a_{0,1}&w_0w_ia_{0,i}\\
w_0w_1a_{0,1}&2w_1^2&w_1w_ia_{1,i}\\
w_0w_ja_{0,j}&w_1w_ja_{1,j}&w_iw_ja_{i,j}
\end{bmatrix}
\nexteq
\det\begin{bmatrix}
2w_0^2&w_0^2+w_1^2&w_0^2+w_i^2\\
w_0^2+w_1^2&2w_1^2&w_1^2+w_i^2\\
w_0^2+w_j^2&w_1^2+w_j^2&w_iw_ja_{i,j}
\end{bmatrix}
\nexteq
-(w_1^2-w_0^2)^2(w_iw_ja_{i,j}-w_i^2-w_j^2).
\end{align*}
Thus (\ref{az2}) holds.

Conversely, suppose that $\{w_i\}_{i=0}^d$ satisfy 
(\ref{az2}). 
Setting $(i,j)=(0,1)$ in (\ref{az2}) gives (\ref{az1d01}).
Setting $i=1$ and replacing $j$ by $i$ in (\ref{az2}) gives
\[
\frac{w_i}{w_1}+\frac{w_1}{w_i}=a_{1,i}.
\]
By (\ref{az2}), we have
\begin{align*}
a_{1,i}w_1w_i&=w_1^2+w_i^2
\nexteq
a_{0,i}w_0w_i+w_1^2-w_0^2,
\end{align*}
which gives (\ref{azid}).

Finally, assume that $a_{i,j}$ ($0\leq i<j\leq d$) are all real
and that (\ref{-22}) holds for $(i_0,i_1)=(0,1)$.
Then by (\ref{az1d}), 
${w_1}/{w_0}$ is an
imaginary number with absolute value $1$. This means
\begin{equation}\label{z01}
\left(\frac{w_1}{w_0}\right)^{-1}=\overline{
\left(\frac{w_1}{w_0}\right)}.
\end{equation}
For $j\geq2$, set
$(X,Y,Z,z)=(a_{0,1},a_{0,j},a_{1,j},{w_1}/{w_0})$ in Lemma~\ref{lem:1}.
It is easy to check that all the denominators in Lemma~\ref{lem:1}
are nonzero. Moreover,
\begin{align}
f&=0&&\text{(by (\ref{az1d})),}\label{fgw1}\\
g&=0&&\text{(by (\ref{N1})),}\label{fgw2}\\
w&=\frac{w_j}{w_0}&&\text{(by (\ref{azid})),}\label{fgw3}
\end{align}
and
\begin{align}
\overline{w'}&=
\overline{\left(
\frac{\left(\frac{w_1}{w_0}\right)^{-2}-1}{
\left(\frac{w_1}{w_0}\right)^{-1}a_{1,j}-a_{0,j}}
\right)}
\nnexteq
\frac{\left(\frac{w_1}{w_0}\right)^2-1}{
\frac{w_1}{w_0}a_{1,j}-a_{0j}}
&&\text{(by (\ref{z01}))}
\nnexteq
\frac{w_1^2-w_0^2}{w_0(a_{1,j}{w_1}-a_{0,j}{w_0})}
\nnexteq
\frac{w_j}{w_0}
&&\text{(by (\ref{azid}))}
\nnexteq
w
&&\text{(by (\ref{fgw3})).}
\label{fgw4}
\end{align}
Now
\begin{align*}
1&=ww'
&&\text{(by (\ref{c3}), (\ref{fgw1}), (\ref{fgw2}))}
\nexteq
|w|^2
&&\text{(by (\ref{fgw4}))}
\nexteq
\left|\frac{w_j}{w_0}\right|^2
&&\text{(by (\ref{fgw3})).}
\end{align*}
Therefore $|w_i|=|w_j|$ for $0\leq i<j\leq d$.
\end{proof}

In the following theorem, $\mathbb{C}^{d(d+1)/2}$
will mean the $\mathbb{C}$-vector space whose coordinates are indexed by
ordered pairs $(i,j)$ of integers with $0\leq i<j\leq d$.

\begin{thm}\label{thm:map}
Let $d,N,N_3,N_4$ be as in Lemma~{\rm \ref{lem:3}}. Define
$\phi:(\C^\times)^{d+1}\to\C^{d(d+1)/2}$ by
\[
\phi(w_0,\dots,w_d)=
\left(\frac{w_i}{w_j}+\frac{w_j}{w_i}\right)_{
0\leq i<j\leq d}.
\]
Then the image of $\phi$ coincides with the zeros of the
ideal generated by the polynomials 
\begin{align}
g(X_{i,j},X_{j,k},X_{i,k})&\quad(\{i,j,k\}\in N_3),\label{N1X}\\
h(X_{i,j},X_{i,k},X_{i,\ell},X_{j,k},X_{j,\ell},X_{k,\ell})&
\quad(\{i,j,k,\ell\}\in N_4),\label{N4X}
\end{align}
where $X_{i,j}=X_{j,i}$.
\end{thm}
\begin{proof}
In the polynomial ring $\C[X_{i,j}\mid 0\leq i<j\leq d]$,
let $I$ denote the ideal generated by (\ref{N1X}) and
(\ref{N4X}). 
By Lemmas \ref{lem:g} and \ref{lem:h}, the image of $\phi$
is contained in the set of zeros of $I$. 

Conversely, let $a=(a_{i,j})_{0\leq i<j\leq d}$ be a zero of $I$.
If there exists $i_0,i_1$ with $0\leq i_0<i_1\leq d$ and
$a_{i_0,i_1}\neq\pm2$, then Lemma~\ref{lem:3} implies that
$a$ is in the image of $\phi$.
Next suppose that 
$a_{i,j}\in\{\pm2\}$ for $0\leq i<j\leq d$.
Then
\[
a_{0,1}=\frac{a_{0,1}}{2}+\frac{2}{a_{0,1}}.
\]
Also,
since $g(a_{0,i},a_{0,j},a_{i,j})=0$, 
we have $a_{0,i}a_{0,j}a_{i,j}=8$. Thus
\begin{align*}
a_{i,j}&=
\frac{8}{a_{0,i}a_{0,j}}\cdot\frac{a_{0,i}^2}{4}\cdot\frac{a_{0,j}^2}{4}
\nexteq
\frac12a_{0,i}a_{0,j}
\nexteq
\frac{a_{0,i}}{a_{0,j}}+\frac{a_{0,j}}{a_{0,i}}.
\end{align*}
Therefore, $a=\phi(2,a_{0,1},\dots,a_{0,d})$ is in the image of $\phi$.
\end{proof}

The following lemma will be used in the proof of Theorem~\ref{thm:main}.

\begin{lem}\label{lem:w1w2+w3}
Let $\mathbb{Q}(X_1,X_2,X_3)$ be the rational function field with three
indeterminates $X_1,X_2,X_3$. Set $X_0=1$ and
\[
x_{i,j}=\frac{X_i}{X_j}+\frac{X_j}{X_i}\quad(0\leq i<j\leq3).
\]
Then 
\begin{align*}
&(X_1X_2X_3+1)(x_{0,1}x_{0,2}+x_{0,3}-x_{1,2})
\\&=
(X_1X_2+X_3)\left(x_{0,1}x_{0,2}x_{0,3}+2-
\frac12(x_{1,2}x_{0,3}+x_{1,3}x_{0,2}+x_{2,3}x_{0,1})\right).
\end{align*}
\end{lem}
\begin{proof}
Straightforward.
\end{proof}

\section{Type-II matrices contained in a Bose--Mesner
algebra}\label{sec:3}

Throughout this section, we let $\mathcal{A}$ denote a
symmetric Bose--Mesner algebra with adjacency matrices
$A_0=I,A_1,\dots,A_d$. Let $n$ be the size of the matrices $A_i$,
and we denote by 
\[
P=(P_{i,j})_{\substack{0\leq i\leq d\\ 0\leq j\leq d}}
\]
the first eigenmatrix of $\mathcal{A}$. Then the 
adjacency matrices are expressed as
\[
A_j=\sum_{i=0}^d P_{i,j}E_i\quad(j=0,1,\dots,d),
\]
where $E_0=\frac{1}{n}J,E_1,\dots,E_d$ are the primitive idempotents
of $\mathcal{A}$.
The second eigenmatrix 
\[
Q=(Q_{i,j})_{\substack{0\leq i\leq d\\ 0\leq j\leq d}}
\]
is defined as $Q=nP^{-1}$, so that
\[
E_j=\frac{1}{n}\sum_{i=0}^d Q_{i,j}A_i\quad(j=0,1,\dots,d)
\]
holds. Since $QP=nI$ and $Q_{i,0}=P_{i,0}=1$ for $i=0,1,\dots,d$, we have
\begin{equation}\label{eq:Qrow}
\sum_{j=1}^d Q_{i,j}=n\delta_{i,0}-1.
\end{equation}

\begin{lem}\label{lem:I}
Let $w_0,w_1,\dots,w_d$ be nonzero complex numbers, and
set
\begin{equation}\label{eq:W} 
W=\sum_{j=0}^dw_j A_j\in\mathcal{A}.
\end{equation}
Set
\begin{equation}  \label{eq:0316}
\beta_k=\sum_{j=0}^d w_j P_{k,j}, \quad
\beta'_k=\sum_{j=0}^dw_j^{-1} P_{k,j} \quad (k=0,1,\ldots,d).
\end{equation}
Then the following are equivalent.
\begin{enumerate}
\item\label{I1}
$W$ is a type-II matrix,
\item\label{I2}
$\beta_k\beta'_k=n$ for all $k$ with $1\leq k\leq d$.
\end{enumerate}
\end{lem}
\begin{proof}
Since
\[
W=\sum_{k=0}^d\beta_k E_k, \quad
W^{(-)}=\sum_{k=0}^d\beta'_k E_k.
\]
$W$ is type-II matrix if and only if
$\beta_k\beta'_k=n$ for $0\leq k\leq d$. 
In particular, \ref{I1} implies \ref{I2}.

To prove the converse, it suffices to show $\beta_0\beta'_0=n$.
Since $W$ is symmetric, the diagonal entries of $WW^{(-)}$ are
all $n$. Thus
\begin{align*}
n^2&=
\tr WW^{(-)}
\nexteq
\sum_{k=0}^d\beta_k\beta'_k\tr E_k
\nexteq
\beta_0\beta'_0+\sum_{k=1}^d n\tr E_k
\nexteq
\beta_0\beta'_0+n\tr (I-E_0)
\nexteq
\beta_0\beta'_0+n(n-1),
\end{align*}
and hence $\beta_0\beta'_0=n$.
\end{proof}

\begin{lem}\label{lem:I1}
Let $X_{i,j}$ $(0\leq i<j\leq d)$ be indeterminates and
let $e_k$ be the polynomial defined by
\begin{equation} \label{eq:04}
e_k=
\sum_{0\leq i<j\leq d} P_{k,i}P_{k,j}X_{i,j}
+\sum_{i=0}^d P_{k,i}^2-n
\quad(k=1,\dots,d).
\end{equation}
Let $a_{i,j}$ $(0\leq i,j\leq d,\;i\neq j)$ and
$w_i$ $(0\leq i\leq d)$ be complex numbers.
Assume that $w_i\neq0$ for all $i$ with $0\leq i\leq d$
and that {\rm(\ref{az2})} holds.
Then the following statements are equivalent:
\begin{enumerate}
\item 
the matrix $W$ given by {\rm(\ref{eq:W})}
is a type-II matrix,
\item
$(a_{i,j})_{0\leq i<j\leq d}$
is a common zero of $e_k$ $(k=1,\dots,d)$.
\end{enumerate}
Moreover, if one of the two equivalent conditions
{\rm(i), (ii)} is satisfied, 
$a_{i,j}\in\R$ for all $i,j$ with $0\leq i<j\leq d$,
and {\rm(\ref{-22})} holds for some $i_0, i_1$ with
$0\leq i_0<i_1\leq d$, then 
$W$ is a scalar multiple of a complex
Hadamard matrix.
\end{lem}
\begin{proof}
Define $\beta_k,\beta'_k$ by (\ref{eq:0316}). Then
\begin{equation}
\beta_k\beta'_k
=\sum_{0\leq i<j\leq d}P_{k,i}P_{k,j} 
\left(\frac{w_{i}}{w_{j}}
+\frac{w_{j}}{w_{i}}\right)
+\sum_{i=0}^d P_{k,i}^2
\quad(k=1,\dots,d).
\label{eq:ei}
\end{equation}
Thus, equivalence of {\rm (i)} and {\rm (ii)} follows immediately from
Lemma~\ref{lem:I}.

Moreover,
suppose that 
$a_{i,j}\in\R$ for all $i,j$ with $0\leq i,j\leq d$, $i\neq j$,
and (\ref{-22}) holds for some $i_0, i_1$.
Then (\ref{a1}) holds. This, together with (\ref{az2}) enables
us to use Lemma~\ref{lem:3} to conclude that 
$w_0,w_1,\dots,w_d$ have the same absolute value.
Therefore $W$ is a scalar multiple of a complex Hadamard matrix.
\end{proof}

\section{Infinite families of complex Hadamard matrices}\label{sec:4}
Let $q\geq4$ be an integer, and $n=q^2-1$.
We consider a three-class association scheme 
$\mathcal{X}=(X,\{R_i\}_{i=0}^{3})$ with the first eigenmatrix:
\begin{equation}\label{P3}
P=\begin{bmatrix}
1&\frac{q^2}{2} - q&\frac{q^2}{2}&q - 2\\
1&\frac{q}{2}&-\frac{q}{2}&-1\\
1&-\frac{q}{2} + 1&-\frac{q}{2}&q - 2\\
1&-\frac{q}{2}&\frac{q}{2}&-1
\end{bmatrix}.
\end{equation}
For $q=2^s$ with an integer $s\geq 2$,
there exists a $3$-class
association scheme with the first eigenmatix (\ref{P3})
(see \cite[12.1.1]{BCN}).

Let $\mathcal{M}=\langle A_0,A_1,A_2,A_3\rangle$
be the Bose--Mesner algebra of $\mathcal{X}$.
Then, $\mathcal{X}$ has two non-trivial fusion schemes.
One is an imprimitive scheme $\mathcal{X}_1=(X,\{R_0, R_1\cup R_2,R_3\})$
with the first eigenmatrix:
\begin{equation} \label{P1}
P_1=\begin{bmatrix}
1&q(q-1)&q-2\\
1&0&-1\\
1&-q+1&q-2
\end{bmatrix}.
\end{equation}
Another is a primitive scheme $\mathcal{X}_2=(X,\{R_0, R_1\cup R_3,R_2\})$ with the first eigenmatrix:
\begin{equation} \label{P2}
P_2=\begin{bmatrix}
1&\frac{q^2}{2}-2&\frac{q^2}{2}\\
1&\frac{q}{2}-1&-\frac{q}{2}\\
1&-\frac{q}{2}-1&\frac{q}{2}
\end{bmatrix}.
\end{equation}

\begin{thm}\label{thm:main}
Let $w_1,w_2,w_3$ be nonzero complex numbers.
The matrix 
\begin{equation} \label{eq:0712}
W=A_0+w_1 A_1+w_2 A_2+w_3 A_3\in\mathcal{M}
\end{equation}
is a type-II matrix if and only if one of the following holds:
\begin{enumerate}
\item\label{t1}
$w_{1}=w_{2}=w_{3}$, where
\[
w_{3}+\frac{1}{w_{3}}+q^2-3=0,
\]
\item\label{t2} 
$w_3$ is as in \ref{t1}, and 
\[
w_1=w_2=\frac{-(q-3)w_3+(q-1)}{q^2-2q-1},
\]
\item\label{t3}
\[
w_1+\frac{1}{w_1}=\frac{2(q^2-6)}{q^2-4},\quad
w_2=-1,\quad
w_3=w_1,
\]
\item\label{t4}
\[
w_1=w_3=1,\quad
w_2+\frac{1}{w_2}=\frac{-2(q^2-2)}{q^2},
\]
\item\label{t5}
\[
w_1+\frac{1}{w_1}=-\frac{2}{q},\quad
w_2=\frac{1}{w_1},\quad
w_3=1,
\]
\item\label{t6}
\[
w_1+\frac{1}{w_1}=a_{0,1},
\]
and
\[
w_i=\frac{w_{1}^2-1}{a_{1,i}w_{1}-a_{0,i}}
\quad(i=2,3),
\]
where
\begin{align*}
a_{0,1}&=\frac{-(q-1)(q-2)+(q+2)r}{2q(q+1)}, \\
a_{0,2}&=\frac{(q+2)(q-1)-(q-2)r}{2q(q-3)}, \\
a_{0,3}&=\frac{5q^2-2q-19-(q-1)r}{2(q+1)(q-3)}, \\
a_{1,2}&=\frac{2(-q^4+2q^3+4q^2-10q+1+(q-1)r)}{q^2(q+1)(q-3)}, \\
a_{1,3}&=-a_{0,2}, \\
r^2&=(17q-1)(q-1).
\end{align*}
Note that $w_1w_2=-w_3$ holds.
\end{enumerate}
\end{thm}
\begin{proof}
Let $K=\mathbb{Q}(r)$, where $r$ is a square 
root of $(17q-1)(q-1)$. Note that $K$ is either $\mathbb{Q}$ or
its quadratic extension (see Remark~\ref{rem:2}). Consider the 
polynomial ring 
\[
R=K[X_{0,1},X_{0,2},X_{0,3},X_{1,2},X_{1,3},X_{2,3}].
\]
We assume that $W$ is a type-II matrix. 
For $i,j\in\{0,1,2,3\}$, define $a_{i,j}$ by (\ref{az2}), where $w_0=1$.
We write 
\[
\boldsymbol{a}=(a_{0,1},a_{0,2},a_{0,3},a_{1,2},a_{1,3},a_{2,3})
\]
for brevity.
Then by Lemmas~\ref{lem:g}, \ref{lem:h} and \ref{lem:I1}, 
$\boldsymbol{a}$ is a common zero of the polynomials
\begin{align*}
& g(X_{i,j},X_{i,k},X_{j,k})\quad(\{i,j,k\}\in\binom{\{0,1,2,3\}}{3}),\\
& h(X_{i,j},X_{i,k},X_{i,l},X_{j,k},X_{j,l},X_{k,l})
\quad(\{i,j,k,l\}=\{0,1,2,3\}), \\
& e_k\quad (k\in\{1,2,3\}).\\
\end{align*}
Let $\mathcal{I}$ be the ideal of $R$ generated by these polynomials.
Then we can verify 
that $\mathcal{I}$ contains the
polynomial $b_1b_2b_3b_4^+b_4^-$, where
\begin{align*}
b_1&=X_{1,2}-2,\\
b_2&=(q^2-4)X_{1,2}+2q^2-12,\\
b_3&=q^2X_{1,2}+2q^2-4,\\
b_4^\pm&=q^2(q+1)(q-3)X_{1,2}-2(-q^4+2q^3+4q^2-10q+1\pm (q-1)r).
\end{align*}

First assume $a_{1,2}$ is a zero of the polynomial $b_1$. 
This implies $w_1=w_2$. 
Let $\mathcal{I}_1$ denote the ideal generated by $\mathcal{I}$ and $b_1$.
Then we can verify that 
$\mathcal{I}_1$ contains the polynomials 
$X_{0,3}+q^2-3$ and $c_1c_2$, where
\begin{align*}
c_1&=X_{0,1}+q^2-3,\\
c_2&=(q^2-2q-1)X_{0,1}-(q^3-3q^2-q+7).
\end{align*}
In particular, 
\begin{equation}\label{a03}
a_{0,3}=w_3+\frac{1}{w_3}=-(q^2-3).
\end{equation}
If $a_{0,1}$ is a zero of the polynomial $c_1$, then
let $\mathcal{I}_1'$ denote the
ideal generated by $\mathcal{I}_1$ and $c_1$.
We can verify that 
$\mathcal{I}_1'$ contains the polynomial $X_{1,3}-2$.
Hence $w_1=w_2=w_3$, and we have Case~\ref{t1}.

If $a_{0,1}$ is a zero of the polynomial $c_2$, then
\begin{equation}\label{a01}
a_{0,1}=\frac{q^3-3q^2-q+7}{q^2-2q-1}.
\end{equation}
Let $\mathcal{I}_1''$ denote the
ideal generated by $\mathcal{I}_1$ and $c_2$.
We can verify that 
\[
(q^2-2q-1)X_{1,3}+(q^3-q^2-q-3)\in\mathcal{I}_1''.
\]
This implies
\begin{equation}\label{a13}
a_{1,3}=-\frac{q^3-q^2-q-3}{q^2-2q-1}\neq\pm2.
\end{equation}
Applying Lemma~\ref{lem:3} with (\ref{a03})--(\ref{a13}),
we obtain from (\ref{azid}),
\begin{align*}
w_1=w_2&=
\frac{w_3^2-1}{a_{1,3}w_3-a_{0,1}}
\nexteq
\frac{-(q-3)w_3+(q-1)}{q^2-2q-1}.
\end{align*}
This gives Case~\ref{t2}.

Next assume $a_{1,2}$ is a zero of the polynomial $b_2$. 
Let $\mathcal{I}_2$ denote the ideal generated by $\mathcal{I}$ and $b_2$.
Then we can verify that 
$\mathcal{I}_2$ contains the polynomials $X_{0,2}+2$, 
$X_{1,3}-2$, and $(q^2-4)X_{0,1}-2(q^2-6)$.
Hence $w_2=-1$, $w_1=w_3$ and
$w_1 + 1/w_1 = 2(q^2-6)/(q^2-4)$. This gives Case~\ref{t3}.

Next assume $a_{1,2}$ is a zero of the polynomial $b_3$. 
Let $\mathcal{I}_3$ denote the ideal generated by $\mathcal{I}_3$ and $b_3$.
Then we can verify that 
$\mathcal{I}_3$ contains the polynomials $X_{0,3}-2$ and $c_4c_5$, where
\begin{align*}
c_4&=q^2X_{0,2}+2(q^2-2),\\
c_5&=qX_{0,2}+2.
\end{align*}
In particular, $w_3=1$.

If $a_{0,2}$ is a zero of the polynomial $c_4$, then
\[
a_{0,2}=-\frac{2(q^2-2)}{q^2}.
\]
Let $\mathcal{I}_3'$ denote the
ideal generated by $\mathcal{I}$ and $c_4$.
We can verify that 
$\mathcal{I}_3'$ contains the polynomial $X_{0,1}-2$.
Hence $w_1=w_3=1$, $w_2+1/w_2+2(q^2-2)/q^2=0$,
and we have Case~\ref{t4}.

If $a_{0,2}$ is a zero of the polynomial $c_5$, then
\begin{equation}\label{a02}
a_{0,2}=-\frac{2}{q}.
\end{equation}
Let $\mathcal{I}_3''$ denote the
ideal generated by $\mathcal{I}_3$ and $c_5$.
We can verify that 
$\mathcal{I}_3''$ contains the polynomials $qX_{0,1}+2$
and $q^2X_{1,2}+2(q^2-2)$. Thus
\begin{align}
a_{0,1}&=-\frac{2}{q},\label{a01-5}\\
a_{1,2}&=-\frac{2(q^2-2)}{q^2}.\label{a12}
\end{align}
Since
\begin{align*}
w_1w_2+\frac{1}{w_1w_2}
&=
\left(w_1+\frac{1}{w_1}\right)
\left(w_2+\frac{1}{w_2}\right)
-\left(\frac{w_2}{w_1}+\frac{w_1}{w_2}\right)
\nexteq
a_{0,1}a_{0,2}-a_{1,2}
\nexteq
2
\end{align*}
by (\ref{a02})--(\ref{a12}), we have $w_1w_2=1$.
This gives Case~\ref{t5}.

Next assume $a_{1,2}$ is a zero of the polynomial $b_4^+b_4^-$.
Without loss of generality,
we may assume $a_{1,2}$ is a zero of the polynomial $b_4^+$.
Let $\mathcal{I}_4$ denote the ideal generated by $\mathcal{I}$ and $b_4^+$.
Then we can verify that 
$\mathcal{I}_4$ contains the following polynomials:
\begin{align}
 & 2q(q+1)X_{0,1}+(q-1)(q-2)-(q-2)r, \\
 & 2q(q-3)X_{0,2}-(q+2)(q-1)+(q-2)r, \\
 & 2(q+1)(q-3)X_{0,3}-(5q^2-2q-19)+(q-1)r, \\
 & q^2(q+1)(q-3)X_{1,2}-2(-q^4+2q^3+4q^2-10q+1+(q-1)r), \\
 & 2q(q-3)X_{1,3}+(q+2)(q-1)-(q-2)r.
\end{align}
Since $a_{1,2}\neq\pm2$, we can apply Lemma~\ref{lem:3} to obtain
Case~\ref{t6} from (\ref{azid}).
Moreover, we can verify that
$\mathcal{I}_4$ contains the polynomial 
\[X_{0,1}X_{0,2}+X_{0,3}-X_{1,2}.\]
Thus
\begin{align*}
0&=
(w_1w_2w_3+1)(a_{0,1}a_{0,2}+a_{0,3}-a_{1,2})
\nexteq
(w_1w_2+w_3)(a_{0,1}a_{0,2}a_{0,3}+2-
\frac12(a_{1,2}a_{0,3}+a_{1,3}a_{0,2}+a_{2,3}a_{0,1}))
\end{align*}
by Lemma~\ref{lem:w1w2+w3}.
Since the ideal generated by $\mathcal{I}_4$ and the polynomial
\[
X_{0,1}X_{0,2}X_{0,3}+2-
\frac12(X_{1,2}X_{0,3}+X_{1,3}X_{0,2}+X_{2,3}X_{0,1})
\]
is trivial, 
we conclude 
$w_1w_2+w_3=0$.

Conversely, assume that 
$w_1, w_2$, and $w_3$ are given in Theorem~\ref{thm:main}.
Then, we show that the matrix $W$ given in (\ref{eq:0712}) is a type-II matrix.
To do this, from Lemma~\ref{lem:I1} we only check that 
$\boldsymbol{a}$ defined by (\ref{az2})
is a zero of the polynomials (\ref{eq:04}).

First consider Case~\ref{t1}.
Then, by (\ref{az2}) we have
\[
\boldsymbol{a}=(-q^2+3,-q^2+3,-q^2+3,2,2,2),
\]
and this is a 
zero of the polynomials (\ref{eq:04}).
Hence $W$ is a type-II matrix.

Next consider Case~\ref{t2}.
This means $w_1=aw_3+b$, with
\begin{align*}
a&=-\frac{q-3}{q^2-2q-1},\\
b&=\frac{q-1}{q^2-2q-1}.
\end{align*}
We claim
\[
\frac{1}{w_1}=a\frac{1}{w_3}+b.
\]
Indeed, this can be checked by showing
\[
(aw_3+b)\left(a\frac{1}{w_3}+b\right)=1
\]
using $w_3+\frac{1}{w_3}=-(q^2-3)$.
Thus
\begin{align*}
a_{0,1}&=w_1+\frac{1}{w_1}
\\&=\frac{q^3-3q^2-q+7}{q^2-2q-1},\\
a_{1,3}&=\frac{w_1}{w_3}+\frac{w_3}{w_1}
\\&=\frac{aw_3+b}{w_3}+w_3\left(a\frac{1}{w_3}+b\right)
\\&=2a+b\left(w_3+\frac{1}{w_3}\right)
\\&=2a-b(q^2-3)
\\&=\frac{-q^3+q^2+q+3}{q^2-2q-1}.
\end{align*}
Now
\[
\boldsymbol{a}=(a_{0,1},a_{0,1},-(q^2-3),2,a_{1,3},a_{1,3})
\]
is a zero of the polynomials (\ref{eq:04}).
Hence $W$ is a type-II matrix.

Next consider Case~\ref{t3}.
Then, by (\ref{az2}) we have
\[
\boldsymbol{a}=\left(\frac{2(q^2-6)}{q^2-4},-2,\frac{2(q^2-6)}{q^2-4},-\frac{2(q^2-6)}{q^2-4},
2,-\frac{2(q^2-6)}{q^2-4}\right),
\]
and this is a 
zero of the polynomials (\ref{eq:04}).
Hence $W$ is a type-II matrix.

Next consider Case~\ref{t4}.
Then, by (\ref{az2}) we have
\[
\boldsymbol{a}=\left(2,\frac{-2(q^2-2)}{q^2},2,\frac{-2(q^2-2)}{q^2},2,\frac{-2(q^2-2)}{q^2}\right),
\]
and this is a 
zero of the polynomials (\ref{eq:04}).
Hence $W$ is a type-II matrix.

Next consider Case~\ref{t5}.
Then, by (\ref{az2}) we have
\begin{align*}
a_{1,2}&=a_{0,1}^2-2 \\
&=\frac{-2(q^2-2)}{q^2},
\end{align*}
and
\[
\boldsymbol{a}=\left(-\frac{2}{q},-\frac{2}{q},2,\frac{-2(q^2-2)}{q^2},-\frac{2}{q},-\frac{2}{q}\right),
\]
and this is a 
zero of the polynomials (\ref{eq:04}).
Hence $W$ is a type-II matrix.

Lastly consider Case~\ref{t6}.
Then
$a$ is a zero of the polynomials (\ref{eq:04}).
Hence $W$ is a type-II matrix.
\end{proof}

\begin{cor}  \label{cor}
Let $W$ be a type-II matrix in {\rm Theorem~\ref{thm:main}}.
Then, $W$ is a complex Hadamard matrix if and only if 
$W$ is given in \ref{t3}, \ref{t4}, \ref{t5}, or \ref{t6}
with $r=\sqrt{(17q-1)(q-1)}>0$.
\end{cor}
\begin{proof}
Since $q\geq 4$, we have $n\geq 15$. Hence, $|w_3|\ne 1$ for Cases \ref{t1} 
and \ref{t2}.
For Cases \ref{t3}, \ref{t4} and \ref{t5}, it is easy to see that
\[
-2<w_i+\frac{1}{w_i}< 2 \quad \text{for some} \ i\in\{1,2,3\}.
\]
Therefore, by Lemma~\ref{lem:I1} the Cases 
\ref{t3}, \ref{t4} and \ref{t5} give complex Hadamard matrices.

Lastly, we consider the Case \ref{t6}.
If $r>0$, then we have 
$0<a_{0,1}<2$.
In fact, since $r>3q$, we have
$a_{0,1}>0$.
Also,
\[
a_{0,1}-2
=\frac{-8q(q+1)(q-3)^2}{2q(q+1)(5q^2+q+2+(q+2)r)}<0.
\]
Thus, we have $0<a_{0,1}< 2$.
By Lemma~\ref{lem:I1}, 
$W$ is a complex Hadamard matrix.

If $r<0$, then
\[
a_{0,1}+2
=\frac{-8q(q+1)(q^2-5)}{2q(q+1)(3q^2+7q-2-(q+2)r)}<0.
\]
This implies $|w_1|\neq1$. Therefore, $W$ is not a complex Hadamard matrix.
\end{proof}

Chan \cite{Chan}, found three complex Hadamard matrices
on the line graph of the Petersen graph.
This is the $3$-class association scheme with the first eigenmatrix (\ref{P3}),
where $q=4$, and the three matrices can be described as 
the matrix $W$ in (\ref{eq:0712}) with $w_1,w_2,w_3$ given as follows.
\begin{align}
w_1&=1,& w_2&=\frac{-7\pm\sqrt{15}i}{8},& w_3&=1, \label{chan1} \\
w_1&=\frac{5\pm\sqrt{11}i}{6},& w_2&=-1,& w_3&=w_1,& \label{chan2} \\
w_1&=\frac{-1\pm\sqrt{15}i}{4},& w_2&=w_1^{-1},& w_3&=1.\label{chan3} 
\end{align}
The cases (\ref{chan1}), (\ref{chan2}) and (\ref{chan3}) 
are given by \ref{t4}, \ref{t3} and \ref{t5}, respectively, of Theorem~\ref{thm:main}.
Note that (\ref{chan1}) is equivalent to the matrix $U_{15}$ in \cite{Szollosi10}.

The complex Hadamard matrix of order $15$ constructed in Theorem~\ref{thm:main}~(vi)
seems to be new. This is obtained
by setting $q=4$ and $r=\sqrt{201}$, and has coefficients $w_1,w_2,w_3
=-w_1w_2$, where
\begin{align*}
w_1+\frac{1}{w_1}&=a_{0,1},\\
a_{0,1}&=\frac{3}{20}(\sqrt{201}-1),\\
w_2&=\frac{a_{0,1}w_1-2}{a_{1,2}w_1-a_{0,2}},\\
a_{0,2}&=-\frac{1}{4}(\sqrt{201}-9),\\
a_{1,2}&=\frac{3\sqrt{201}-103}{40}.
\end{align*}
We have verified using the span condition \cite[Proposition 4.1]{Nic}
that, this matrix, as well as the one given by (\ref{chan1}) are
isolated, while the two matrices given by (\ref{chan2}) and (\ref{chan3})
do not satisfy the span condition.

\begin{rem}  \label{rem:1}
None of the complex Hadamard matrices given in Corollary~\ref{cor}
is a Butson--Hadamard matrix, that is, entries are roots of unity.
This follows from the fact that $a_{i,j}$ is not an algebraic
integer for some $i,j$. Indeed, for the Case \ref{t3},
\[
a_{0,1}=2-\frac{4}{(q-2)(q+2)}
\]
is not an integer.

As for the Case \ref{t4},
\[
a_{0,2}=-2+\frac{4}{q^2}
\]
is not an integer.

As for the Case \ref{t5},
\[
a_{0,1}=\frac{-2}{q}
\]
is not an integer.

Finally, for the Case \ref{t6}, first suppose that $r$ is irrational.
If $a_{0,1}$ is an algebraic integer, then so is
\[
a_{0,1}'=\frac{-(q-1)(q-2)-(q+2)r}{2q(q+1)}
\]
But then
\[
a_{0,1}+a_{0,1}'=-1+\frac{4q-2}{q(q+1)}
\]
is an integer, which is absurd.

Next suppose that $r$ is rational. Then $a_{0,1}$ is an integer,
so we have $a_{0,1}\in\{0,\pm1,\pm2\}$. But this does not occur.
\end{rem}

\begin{rem}  \label{rem:2}
Let $q$ be an even positive integer with $q\geq4$. Then
$\sqrt{(q-1)(17q-1)}$ is an integer if and only if
\begin{align*}
q&\in
\left\{\frac{1}{34}\left(\tr\left((2177+528\sqrt{17})^{n}(433+105\sqrt{17})\right)+18\right)\mid n\in\Z\right\}
\\&=\{
\ldots,41210,10,26,110890,482812730,\ldots\},
\end{align*}
where $\tr$ denotes the trace from $\Q(\sqrt{17})$ to
$\Q$, i.e., $\tr(a+b\sqrt{17})=2a$ for $a,b\in\Q$.
In this case, the entries of the matrix $W$ in
Theorem~\ref{thm:main}(vi) are quadratic irrationals, so there are two pairs of
algebraically conjugate matrices arising from two choices for $r$.
If $\sqrt{(17q-1)(q-1)}$ is not an integer, then the four matrices
in Theorem~\ref{thm:main}(vi) are algebraically conjugate to each other.
\end{rem}

\section{Equivalence}\label{sec:5}
For a type-II matrix $W$ of order $n$,
the Haagerup set $H(W)$ (see \cite{Haagerup}) is defined as
\[
H(W)= \left\{\frac{W_{i_1,j_1}W_{i_2,j_2}}{W_{i_1,j_2}W_{i_2,j_1}}
\Bigm| 1\leq i_1,i_2,j_1,j_2\leq n\right\}.
\]
We also define
\[
K(W)=\left\{w+\frac{1}{w} \Bigm| w\in H(W)\setminus\{1\}\right\}.
\]
If $W_1$ and $W_2$ are equivalent, then clearly $H(W_1)=H(W_2)$,
and hence $K(W_1)=K(W_2)$.
In this section, we compute the Haagerup sets of type-II matrices
constructed in Theorem~\ref{thm:main} to conclude that some of them
are inequivalent to others.

We suppose that
\[
W=\sum_{i=0}^d w_iA_i
\]
is a complex Hadamard matrix, where $A_0,\dots,A_d$ are the adjacency
matrices of a symmetric Bose--Mesner algebra of an
association scheme $(X,\{R_i\}_{i=0}^d)$, and $w_0=1$.
Let $H(W)$ be the Haagerup set of $W$. Then
\[
H(W)=\bigcup_{i=1}^4 H_i(W),
\]
where
\[
H_i(W)=\left\{
\frac{W_{x_1,y_1}W_{x_2,y_2}}{W_{x_2,y_1}W_{x_1,y_2}}\Bigm|
x_1,x_2,y_1,y_2\in X,\;
|\{x_1,x_2,y_1,y_2\}|=i\right\}
\]
for $i=1,2,3,4$.
Clearly,
\begin{align}
H_1(W)&=\{1\},\nonumber \\
H_2(W)&=\{1\}\cup\{w_i^{\pm2}\mid i=1,\ldots,d\}. \label{H2}
\end{align}
It should be remarked that, although $H(W)$ is an invariant,
none of $H_i(W)$ ($i=2,3,4$) is.

\begin{lem}\label{lem:v3}
If $|X|\geq3$, then
\[
H_3(W)=\{1\}\cup
\left\{
\left(\frac{w_iw_j}{w_k}\right)^{\pm1}
\Bigm| 1\leq i,j,k\leq d,\;p_{ij}^k>0
\right\}.
\]
\end{lem}
\begin{proof}
Let
\[
w=\frac{W_{x_1,y_1}W_{x_2,y_2}}{W_{x_2,y_1}W_{x_1,y_2}}\in H_3(W),
\]
where $|\{x_1,x_2,y_1,y_2\}|=3$. If $x_1=x_2$, then $w=1$.
If $x_1=y_1$, then $w=(w_iw_j/w_k)^{-1}$, where
$(x_2,x_1)\in R_i$, $(x_1,y_2)\in R_j$, $(x_2,y_2)\in R_k$, and
$1\leq i,j,k\leq d$.
All other cases can be treated in a similar manner, 
and we conclude that
$H_3(W)$ is contained in the right-hand side. 
Arguing in the opposite way, we obtain the reverse containment.
\end{proof}

\begin{lem}\label{lem:d}
Let $\Delta$ be a subset of $\{1,\dots,d\}$. 
Suppose that there exists $i\in\{1,\dots,d\}$ such that
$p_{i_1,j_1}^i>0$ for any $i_1,j_1\in\Delta$. Then
\[
H_4(W)\supset
\left\{\frac{w_{i_1}w_{i_2}}{w_{j_1}w_{j_2}}\Bigm|
i_1,i_2,j_1,j_2\in\Delta\right\}\setminus\{1\}.
\]
In particular, if there exists $i\in\{1,\dots,d\}$ such that
$p_{i_1,j_1}^i>0$ for any $i_1,j_1\in\{1,\dots,d\}$, then
\[
H_4(W)\setminus\{1\}=
\left\{\frac{w_{i_1}w_{i_2}}{w_{j_1}w_{j_2}}\Bigm|
i_1,i_2,j_1,j_2\in\{1,\dots,d\}\right\}\setminus\{1\}.
\]
\end{lem}
\begin{proof}
Pick $(x_1,x_2)\in R_i$ with $i\neq0$. For any 
$i_1,i_2,j_1,j_2\in\Delta$, there exists $y_1$ such that
$(x_1,y_1)\in R_{i_1}$ and $(x_2,y_1)\in R_{j_1}$. Similarly,
there exists $y_2$ such that
$(x_1,y_2)\in R_{j_2}$ and $(x_2,y_2)\in R_{i_2}$.
If $(i_1,j_1)\not=(j_2,i_2)$, then $y_1\not= y_2$.
Thus $|\{x_1,x_2,y_1,y_2\}|=4$, and hence
\begin{align*}
H_4(W)&\supset 
\left\{
\frac{w_{i_1}w_{i_2}}{w_{j_1}w_{j_2}}
\Bigm| i_1,i_2,j_1,j_2\in\Delta, (i_1,j_1)\not=(j_2,i_2)
\right\} \\
&\supset
\left\{
\frac{w_{i_1}w_{i_2}}{w_{j_1}w_{j_2}}
\Bigm| i_1,i_2,j_1,j_2\in\Delta
\right\}\setminus\{1\}.
\end{align*}
The second statement follows immediately from the first one, since
$H_4(W)\setminus\{1\}$ is contained in the right-hand side by the definition.
\end{proof}

\begin{lem}\label{lem:d-1}
Suppose that there exists $i\in\{1,\dots,d-1\}$ such that
$p_{i_1,j_1}^i>0$ for any $i_1,j_1\in\{1,\dots,d-1\}$. 
Moreover, suppose $p_{i,j}^d>0$ for any $j\in\{1,\dots,d-1\}$.
Then
\[
H_4(W)\setminus\{1\}=
\left\{\frac{w_{i_1}w_{i_2}}{w_{j_1}w_{j_2}}\Bigm|
i_1,i_2,j_1,j_2\in\{1,\dots,d\}\right\}\setminus\{1\}.
\]
\end{lem}
\begin{proof}
Clearly, $H_4(W)\setminus\{1\}$ is contained in the right-hand side, so 
we show the reverse containment. 
Setting $\Delta=\{1,\dots,d-1\}$ in Lemma~\ref{lem:d}, we have
\[
H_4(W)\setminus\{1\}\supset
\left\{\frac{w_{i_1}w_{i_2}}{w_{j_1}w_{j_2}}\Bigm|
i_1,i_2,j_1,j_2\in\{1,\dots,d-1\}\right\}\setminus\{1\}.
\]
It remains to show that
\[
\frac{w_{i_1}w_{d}}{w_{j_1}w_{j_2}}\in H_4(W)
\text{ and }
\frac{w_{d}^2}{w_{j_1}w_{j_2}}\in H_4(W)
\quad(\forall i_1,j_1,j_2\in\{1,\dots,d-1\}).
\]
Pick $(x_1,x_2)\in R_i$. Then by the assumption,
there exists $y_1$ such that
$(x_1,y_1)\in R_{i_1}$ and $(x_2,y_1)\in R_{j_1}$. Similarly,
there exists $y_2$ such that
$(x_1,y_2)\in R_{{j_2}}$ and $(x_2,y_2)\in R_{d}$.
Now
\[
\frac{w_{i_1}w_{d}}{w_{j_1}w_{j_2}}
=\frac{W_{x_1,y_1}W_{x_2,y_2}}{W_{x_2,y_1}W_{x_1,y_2}}\in H_4(W).
\]
Replacing $i_1$ by $d$ in the above argument gives
\[
\frac{w_{d}^2}{w_{j_1}w_{j_2}}
=\frac{W_{x_1,y_1}W_{x_2,y_2}}{W_{x_2,y_1}W_{x_1,y_2}}\in H_4(W).
\]
\end{proof}

Below, we determine the Haagerup set of the type-II matrices
given in Theorem~\ref{thm:main}.
In what follows, let
$\mathfrak{X}=(X,\{R_i\}_{i=0}^3)$ be an association scheme
with the first eigenmatrix (\ref{P3}), where $q$ is an even positive
integer with $q\geq4$.
The intersection numbers of $\mathfrak{X}$ are given by
\begin{align}
B_1&=\begin{bmatrix}
0&1&0&0 \\
\frac{q^2}{2}-q&\frac{(q-2)^2}{4}&\frac{(q-2)^2}{4}&\frac{q(q-4)}{4} \\
0&\frac{q(q-2)}{4}&\frac{q(q-2)}{4}&\frac{q^2}{4} \\
0&\frac{q-4}{2}&\frac{q-2}{2}&0
\end{bmatrix}, \label{BB1}\displaybreak[0]\\
B_2&=\begin{bmatrix}
0&0&1&0 \\
0&\frac{q(q-2)}{4}&\frac{q(q-2)}{4}&\frac{q^2}{4} \\
\frac{q^2}{2}&\frac{q^2}{4}&\frac{q^2}{4}&\frac{q^2}{4} \\
0&\frac{q}{2}&\frac{q-2}{2}&0
\end{bmatrix}, \label{BB2}\displaybreak[0]\\
B_3&=\begin{bmatrix}
0&0&0&1 \\
0&\frac{q-4}{2}&\frac{q-2}{2}&0 \\
0&\frac{q}{2}&\frac{q-2}{2}&0 \\
q-2&0&0&q-3
\end{bmatrix}, \label{BB3}
\end{align}
where $B_h$ has $(i,j)$-entry $p_{hi}^j$ ($0\leq i,j\leq 3$).

\begin{lem} \label{lem:H(W)}
Let $W=I+\sum_{i=1}^3w_iA_i$ be a type-II matrix belonging to
the Bose--Mesner algebra of $\mathfrak{X}$. 
Then
\begin{align*}
H(W)&=\{w_i^{\pm2}\mid i=1,2,3\}
\\&\quad\bigcup
\left\{
\left(\frac{w_{i_1}w_{i_2}}{w_{i_3}}\right)^{\pm1}
\Bigm|
1\leq i_1,i_2,i_3\leq 3,\;
p_{i_2,i_3}^{i_1}>0
\right\}
\\&\quad\bigcup
\left\{\frac{w_{i_1}w_{i_2}}{w_{j_1}w_{j_2}}\Bigm|
i_1,i_2,j_1,j_2\in\{1,2,3\}\right\}.
\end{align*}
\end{lem}
\begin{proof}
Note that
$H_3(W)$ is determined by Lemma~\ref{lem:v3}.
We can also easily see that the assumptions of
Lemma~\ref{lem:d-1} are satisfied by setting $i=2$.
This gives $H_4(W)\setminus\{1\}$. Since
$H_1(W)=\{1\}\subset H_4(W)$, we have the desired expression for $H(W)$.
\end{proof}

Using Lemma~\ref{lem:H(W)}, we can determine the Haagerup set $H(W)$
for each type-II matrix given in Theorem~\ref{thm:main}.
Note that the description of $H(W)$ in Table~\ref{tab:1}
is valid for all even $q\geq 4$, even though $p_{11}^3=0$ for $q=4$.

\begin{table}[h]
\begin{center}
\begin{tabular}{|c|l|l|}
\hline
&$H(W)\setminus\{1\}$&$K(W)$\\
\hline
\ref{t1}&$\{w_1^{\pm1},w_1^{\pm2}\}$&$-q^2+3$, $q^4-6q^2+7$\\
\hline
\ref{t2}&$\{w_1^{\pm1},w_1^{\pm2},
w_3^{\pm1},w_3^{\pm2}$,
&
$-q^2+3$, $q^4-6q^2+7$,\\
&
$(\frac{w_1^2}{w_3})^{\pm1},
(\frac{w_3}{w_1})^{\pm1},
(\frac{w_3}{w_1})^{\pm2}\}$&
$\frac{q^3-3q^2-q+7}{q^2-2q-1}, \ldots$\\
\hline
\ref{t3}&$\{-1,\pm w_1^{\pm1},\pm w_1^{\pm2}\}$
& $-2$, $\pm\frac{2(q^2-6)}{q^2-4}$, 
$\pm\frac{2(q^4-16q^2+56)}{(q^2-4)^2}$\\
\hline
\ref{t4}&$\{w_2^{\pm1},w_2^{\pm2}\}$ & 
$-\frac{2(q^2-2)}{q^2}$, $\frac{2(q^4-8q^2+8)}{q^4}$\\
\hline
\ref{t5}&$\{w_1^{\pm1},w_1^{\pm2},w_1^{\pm3},w_1^{\pm4}\}$ & 
$-\frac{2}{q}, \ldots$\\
\hline
\ref{t6}&$\{-1,\pm w_1^{\pm1},\pm w_2^{\pm1},
\pm w_1^{\pm2}$,& $-2,\frac{-(q-1)(q-2)+(q+2)r}{2q(q+1)},\ldots$ \\
&$\pm w_2^{\pm2},(w_1^{\pm1}w_2^{\pm1})^2,
\pm w_1^{\pm1}w_2^{\pm1}$,& \\
&$\pm(w_1^2w_2)^{\pm1},\pm(w_1w_2^2)^{\pm1}\}$ & 
\\
\hline
\end{tabular}
\end{center}
\caption{Haagerup sets}
\label{tab:1}
\end{table}

The elements of $H(W)$ given in Table~\ref{tab:1}
can be found as follows: \par
As for the Case \ref{t1}, $K(W)$ has two elements
\[
w_1+\frac{1}{w_1}=-q^2+3, \quad
w_1^2+\frac{1}{w_1^2}=q^4-6q^2+7.
\]
As for \ref{t2}, $K(W)$ contains
\[
w_1+\frac{1}{w_1}=\frac{q^3-3q^2-q+7}{q^2-2q-1},
\]
by (\ref{a01}).
The Cases \ref{t3} and \ref{t4} are immediate.
Finally, it is clear that $K(W)$ contains
$-\frac{2}{q}$ and $-2$, in the Cases \ref{t5} and \ref{t6}, respectively.
We do not need the remaining elements of 
$K(W)$ to prove the following propositions.

\begin{prop}  \label{prop:1}
Let $W_1,\ldots, W_6$ be type-II matrices given in 
\ref{t1}--\ref{t6} of
Theorem~{\rm \ref{thm:main}}, respectively.
Then $W_1,\ldots,W_6$ are pairwise inequivalent.
\end{prop}
\begin{proof}
It suffices to show that the sets $K_i=K(W_i)$ ($i=1,\ldots,6$)
are pairwse distinct.
Since $-1\in H(W_i)$ if and only if $i\in\{3,6\}$,
we see that $K_i\not= K_j$ whenever $i\in\{1,2,4,5\}$ and
$j\in\{3,6\}$.
Since $W_4$ and $W_5$ are complex Hadamard matrices by Corollary~\ref{cor},
we see that $K_4$ and $K_5$ are contained in the interval $[-2,2]$,
while neither $K_1$ nor $K_2$ is. Thus 
$K_i\not= K_j$ whenever $i\in\{1,2\}$ and $j\in\{4,5\}$.
Also, we have $K_1\not= K_2$ since
\[
\frac{q^3-3q^2-q+7}{q^2-2q-1}\in K_2\setminus K_1,
\]
and $K_4\not= K_5$ since
\[
-\frac{2}{q}\in K_5\setminus K_4.
\]
It can be checked directly that the element
\[
a_{0,1}=\frac{-(q-1)(q-2)+(q+2)r}{2q(q+1)}\in K_6
\]
does not belong to $K_3$. Thus $K_3\not= K_6$.
\end{proof}

\begin{prop}  \label{prop:2}
Let $W_+$ and $W_-$ be type-II matrices given in 
Theorem~{\rm \ref{thm:main} (vi)} with $r>0$ and $r<0$, respectively.
Then $W_+$ and $W_-$ are inequivalent.
\end{prop}
\begin{proof}
By Corollary~\ref{cor}, $W_+$ is a complex Hadamard matrix.
Thus every element of $H(W_+)$ has absolute value $1$,
and hence $K(W_+)$ is contained in the interval $[-2,2]$.
On the other hand, the element
\[
a_{0,1}=\frac{-(q-1)(q-2)-(q+2)\sqrt{(17q-1)(q-1)}}{2q(q+1)}\in K(W_-)
\]
does not belong to the interval $[-2,2]$.
Thus $K(W_+)\not= K(W_-)$.
\end{proof}

We were able to use the Haagerup set to distinguish some of
the complex Hadamard matrices in Theorem~\ref{thm:main}. This is
because the Haagerup set can be described by the intersection
numbers of the association scheme, and is independent of the
isomorphism class. In general, if $q\geq8$ is a power of $2$,
there may be many non-isomorphic association schemes with the
eigenmatrix (\ref{P3}). We do not know whether complex Hadamard
matrices having the same coefficients are equivalent if they
belong to Bose--Mesner algebras of non-isomorphic association schemes.

Note that there are two type-II matrices
described in Theorem~\ref{thm:main}\ref{t1}, since $w_1=w_2=
w_3$ is either of the two zeros of a quadratic equation.
Similarly, there are two type-II matrices in each of \ref{t2}--\ref{t5}
in Theorem~\ref{thm:main}. Moreover, there are four type-II matrices
in \ref{t6} because there are two choices for $r$ and $a_{0,1}^2-4\neq0$.
The following lemma shows that the two type-II matrices
in Theorem~\ref{thm:main}\ref{t1} are inequivalent,
and so are those in Theorem~\ref{thm:main}\ref{t2}.

\begin{lem}\label{lem:ineq1}
Let $W$ and $W'$ be type-II matrices belonging to the
Bose--Mesner algebra of an association scheme 
$\mathcal{X}=(X,\{R_i\}_{i=0}^d)$.
Suppose that each of $W$ and $W'$ has $d+1$ distinct entries, the valencies of 
$\mathcal{X}$ are pairwise distinct, and
$\min\{p_{11}^i\mid 0<i\leq d\}>\frac{|X|}{2}$.
If $W$ and $W'$ are 
equivalent, then $W$ is a scalar
multiple of $W'$.
\end{lem}
\begin{proof}
Suppose that $W$ and $W'$ are 
equivalent. Then there exist
diagonal matrices $D=\diag(d_1,\dots,d_n)$, $D'=\diag(d'_1,\dots,d'_n)$
and permutation matrices $T=(\delta_{\tau(i),j})$,
$T'=(\delta_{\tau'(i),j})$ such that 
\begin{equation}\label{DWD}
DWD'=TW'{T'}^{-1}
\end{equation}
holds. 
Let
\begin{align*}
W&=\sum_{i=0}^dw_i A_i,\\
W'&=\sum_{i=0}^dw'_i A_i.
\end{align*}
For $x\in X$, define
\[
R_1(x)=\{y\in X\mid (x,y)\in R_1\}.
\]
Let $x_1,x_2\in X$ be distinct. Then
$|R_1(x_1)\cap R_1(x_2)|>\frac{|X|}{2}$ by the assumption.
Similarly,
$|R_1(\tau(x_1))\cap R_1(\tau(x_2))|>\frac{|X|}{2}$, hence
\[|\tau'^{-1}(R_1(\tau(x_1))\cap R_1(\tau(x_2)))|>\frac{|X|}{2}.\]
Therefore, there exists 
\[
y\in R_1(x_1)\cap R_1(x_2)\cap
\tau'^{-1}(R_1(\tau(x_1))\cap R_1(\tau(x_2))).
\]
Equivalently, $y$ satisfies
\[
(x_1,y),(x_2,y),(\tau(x_1),\tau'(y)),(\tau(x_2),\tau'(y))\in R_1.
\]
Comparing $(x_1,y)$- and $(x_2,y)$-entries of (\ref{DWD}), we find
\[
d_{x_1}w_1 d'_y=d_{x_2}w_1 d'_y=w'_1.
\]
This implies $d_{x_1}=d_{x_2}$. Since $x_1,x_2\in X$ were
arbitrary, $D=dI$ for some $d$. Similarly, $D'=d'I$ for some $d'$.
We now have
\[
dd'W=TW'{T'}^{-1}.
\]
Since $W$ has $d+1$ distinct entries, we have
\begin{align*}
|X|k_i&=
|\{(x,y)\in X\times X\mid (TW'{T'}^{-1})_{x,y}=dd'w_i\}|
\nexteq
|\{(x,y)\in X\times X\mid W'_{\tau(x),\tau'(y)}=dd'w_i\}|
\nexteq
|\{(x,y)\in X\times X\mid W'_{x,y}=dd'w_i\}|.
\end{align*}
Since $W'$ has $d+1$ distinct entries and
the valencies are pairwise distinct, we obtain
$dd'w_i=w'_i$. Hence $dd'W=W'$.
\end{proof}

\begin{prop}
Let $W$ be a type-II matrix given
in \ref{t1} of Theorem~{\rm \ref{thm:main}}.
Then $W$ and $W^{(-)}$ are inequivalent.
The same conclusion holds if $W$ is a type-II matrix given
in \ref{t2} of Theorem~{\rm \ref{thm:main}}.
\end{prop}
\begin{proof}
First, let $W$ be a type-II matrix given
in \ref{t1} of Theorem~{\rm \ref{thm:main}}.
Then the matrices $W$ and $W^{(-)}$ belong to the Bose--Mesner
algebra of the association scheme
$(X,\{R_0,R_1\cup R_2\cup R_3\})$. Since 
$W$ has two distinct entries and this association scheme
satisfies the hypothesis of Lemma~\ref{lem:ineq1}, 
$W$ and $W^{(-)}$ are inequivalent.

If $W$ is a type-II matrix given
in \ref{t2} of Theorem~{\rm \ref{thm:main}}, then
$W$ and $W^{(-)}$ belong to the Bose--Mesner
algebra of the association scheme
$(X,\{R_0,R_1\cup R_2,R_3\})$. Since 
$W$ has three distinct entries and this association scheme
satisfies the hypothesis of Lemma~\ref{lem:ineq1}, 
$W$ and $W^{(-)}$ are inequivalent.
\end{proof}

We do not know whether the two type-II matrices
in each of \ref{t3}--\ref{t5} in Theorem~\ref{thm:main}
are equivalent or not, and whether the two type-II matrices
in Theorem~\ref{thm:main}\ref{t6} with a given sign for $r$ 
are equivalent or not.

\section{Nomura algebras}\label{sec:6}
Since $q^2-1$ is a composite, there are uncountably many
inequivalent complex Hadamard matrices of order $q^2-1$
by \cite{Craigen}. Indeed, such matrices can be constructed
using generalized tensor products \cite{HS}. We show that none of our
complex Hadamard matrices is equivalent to a generalized tensor product.
This is done by showing that the Nomura algebra of any of our
complex Hadamard matrices has dimension $2$. 
According to
\cite{HS}, the Nomura algebra of the generalized tensor product
of type-II matrices is imprimitive, and this is never the case
when it has dimension $2$. 

For a type-II matrix $W\in M_X(\C)$ and $a,b\in X$, we define column vectors
$Y_{ab}$ by setting 
\[
(Y_{ab})_x=
\frac{W_{xa}}{W_{xb}}\quad(x\in X).
\]
The {\em Nomura algebra} $N(W)$ of $W$ is the algebra of matrices
in $M_n(\C)$ such that $Y_{ab}$ is an eigenvector for all $a,b\in X$.
It is shown in \cite[Theorem 1]{JMN} that the Nomura algebra is 
a Bose--Mesner algebra. 
We first show that the Nomura algebra consists of 
symmetric matrices.

Throughout this section, we let $N$ denote the Nomura algebra
of any of the type-II matrices given in Theorem~\ref{thm:main}.

\begin{lem}\label{lem:NS}
The algebra $N$ is symmetric.
\end{lem}
\begin{proof}
Suppose that $N$ is not symmetric. Then
by \cite[Proposition 6(i)]{JMN}, there exists $(b,c)\in X^2$ with $b\neq c$
such that
\[
\sum_{x\in X}\frac{W_{x,b}^2}{W_{x,c}^2}=0
\]
This is equivalent to
\[
\sum_{j,k}
p_{jk}^i \frac{w_j^2}{w_k^2}=0
\]
for some $i\in\{1,2,3\}$. Using the notation (\ref{az2}), we have
\begin{align}
\sum_{j,k}p_{jk}^i \frac{w_j^2}{w_k^2}
&=
\sum_{j<k}p_{jk}^i 
\left(\frac{w_j^2}{w_k^2}
+\frac{w_k^2}{w_j^2}\right)
+\sum_{j=0}^3 p_{jj}^i 
\nnexteq
\sum_{j<k}p_{jk}^i 
\left(\left(\frac{w_j}{w_k}
+\frac{w_k}{w_j}\right)^2-2\right)
+\sum_{j=0}^3 p_{jj}^i 
\nnexteq
\sum_{j<k}p_{jk}^i (a_{j,k}^2-2)
+\sum_{j=0}^3 p_{jj}^i. 
\label{eq:NS}
\end{align}
It can be verified by computer that (\ref{eq:NS})
is nonzero for each of the cases (i)--(vi)
in Theorem~\ref{thm:main}.
\end{proof}

Since $N$ is symmetric, the adjacency matrices of $N$ are the
$(0,1)$-ma\-trices representing the connected components of 
the Jones graph defined as follows
(see \cite[Sect.~3.3]{JMN}).
The {\em Jones graph} of a type-II matrix $W\in M_X(\C)$ is the graph
with vertex set $X^2$ such that two distinct vertices $(a,b)$
and $(c,d)$ are adjacent whenever $\langle Y_{ab}, Y_{cd}\rangle\neq0$,
where $\langle \;,\;\rangle$ denotes the ordinary (not Hermitian)
scalar product.

\begin{prop}\label{prop:N2}
The algebra $N$ coincides with the linear span of $I$ and $J$.
In particular, none of the type-II matrices given in 
Theorem~{\rm \ref{thm:main}} is equivalent to a nontrivial generalized
tensor product.
\end{prop}
\begin{proof}
We aim to show that $\{(x,y)\mid x,y\in X,\;x\neq y\}$ is a 
connected component of the Jones graph.
Fix $x,y,z$ with $(x,y),(y,z),(z,x)\in R_3$. 
Then in the Jones graph, $(x,y)$ and $(y,x)$ belong to the
same connected component, and 
$(x,z)$ and $(z,x)$ belong to the same connected component,
by Lemma~\ref{lem:NS}.

We claim that $(x,y)$ and $(x,z)$ belong to the same connected component.
Indeed, if $(x,y)$ and $(x,z)$ belong to different connected components, then
$(y,x)$ and $(z,x)$ belong to different connected components.
In particular, 
\[
\langle Y_{xy},Y_{xz}\rangle=\langle Y_{yx},Y_{zx}\rangle=0.
\]
Let
\[
c_{i,j,k}=|\{u\in X\mid (x,u)\in R_i,\;(y,u)\in R_j,\;(z,u)\in R_k\}|.
\]
Then for $j,k\in\{1,2\}$,
\[
c_{1,j,k}+c_{2,j,k}=
c_{j,1,k}+c_{j,2,k}=
c_{j,k,1}+c_{j,k,2}=p_{jk}^3,
\]
and
\[
\sum_{i,j,k=0}^3 c_{i,j,k}\frac{w_i^2}{w_jw_k}=
\sum_{i,j,k=0}^3 c_{i,j,k}\frac{w_jw_k}{w_i^2}=0.
\]
It can be verified by computer that these conditions
give rise  to a polynomial equation in $q$ which has no solution
in even positive integers. 
Therefore, we have proved the claim.
This implies that, for each equivalence class $C$ of the
equivalence relation $R_0\cup R_3$,
$(C\times C)\cap R_3$ is a clique in the Jones graph.

Let $C$ and $C'$ be two distinct equivalence classes of $R_0\cup R_3$.
We claim that,
for any $(x,z)\in C\times C'$, there exist $y\in C$ and $y'\in C'$ such that
$\langle Y_{xy},Y_{xz}\rangle\neq0$ and
$\langle Y_{y'z},Y_{xz}\rangle\neq0$.
Suppose
$(x,z)\in R_1$ and 
$\langle Y_{xy},Y_{xz}\rangle=0$ for all $y\in C$. Then
\begin{align*}
0&=\sum_{y\in C}\langle Y_{xy},Y_{xz}\rangle
\nexteq
\sum_{y\in C}\sum_{u\in X} (Y_{xy})_u(Y_{xz})_u
\nexteq
\sum_{y\in C}\sum_{u\in X} \frac{W_{xu}^2}{W_{yu}W_{zu}}
\nexteq
\sum_{y\in C}\sum_{i,j=0}^3 \sum_{u\in R_i(x)\cap R_j(z)}
\frac{W_{xu}^2}{W_{yu}W_{zu}}
\nexteq
\sum_{i,j=0}^3 \sum_{u\in R_i(x)\cap R_j(z)}
\sum_{k=0}^3 \sum_{y\in C\cap R_k(u)}
\frac{w_i^2}{w_kw_j}
\nexteq
\sum_{i,j=0}^3 \sum_{u\in R_i(x)\cap R_j(z)}
\sum_{k=0}^3 p_{3k}^i \frac{w_i^2}{w_kw_j}
\nexteq
\sum_{i,j,k=0}^3 p_{ij}^1
p_{3k}^i \frac{w_i^2}{w_kw_j}.
\end{align*}
It can be verified by computer that this leads to a polynomial
equation in $q$ which has no solution in even positive integers. 
Similarly, suppose $(x,z)\in R_2$ and
$\langle Y_{xy},Y_{xz}\rangle=0$ for all $y\in C$. Then
\[
\sum_{i,j,k=0}^3 p_{ij}^2
p_{3k}^i \frac{w_i^2}{w_kw_j}=0,
\]
and again this leads to a contradiction.
Thus, there exists $y\in C$
such that $\langle Y_{xy},Y_{xz}\rangle\neq0$.
Switching the role of $x$ and $z$, we see that there exists
$y'\in C'$ such that $\langle Y_{y'z},Y_{xz}\rangle\neq0$.
Therefore, we have proved the claim.

Since $C$ and $C'$ are arbitrary,
the claim shows that, in the Jones graph, $R_3$ is contained in a connected component,
and that every element $(x,z)\in R_1\cup R_2$ is adjacent to an element
of $R_3$. Thus,
$R_1\cup R_2\cup R_3$ is a connected component of the Jones graph. Therefore, $\dim N=2$.
\end{proof}

\subsection*{Acknowledgements}
We are very grateful to Ferenc Sz\"oll\H{o}si and Ada Chan for helpful
discussions on various parts of the paper.
We also thank Doug Leonard for giving us suggestions for
Section~\ref{sec:2}
 and Takao Komatsu for consultation on Pell equations.

\newpage

\appendix
\section{Verification by Magma}

\subsection*{Proof of Lemma \ref{lem:g}}
\begin{verbatim}
F<X,Y,Z>:=FunctionField(Rationals(),3);
g:=X^2+Y^2+Z^2-X*Y*Z-4;
g @ hom<F->F|X/Y+Y/X,X/Z+Z/X,Z/Y+Y/Z> eq 0;
\end{verbatim}

\subsection*{Proof of Lemma \ref{lem:1}}
\begin{verbatim}
F4<X,Y,Z,z>:=FunctionField(Rationals(),4);
f:=z^2-z*X+1;
g:=X^2+Y^2+Z^2-X*Y*Z-4;
w:=(z^2-1)/(z*Z-Y);
ww:=(z^(-2)-1)/(z^(-1)*Z-Y);
w*ww eq 1+(z^2*g+(2*z*X-z*Y*Z+f)*f)/(z*(z*Z-Y)*(z*Y-Z));
\end{verbatim}

\subsection*{Proof of Lemma \ref{lem:h}}
\begin{verbatim}
F<X0,X1,X2,X3>:=FunctionField(Rationals(),4);
x01:=X0/X1+X1/X0;
x02:=X0/X2+X2/X0;
x03:=X0/X3+X3/X0;
x12:=X1/X2+X2/X1;
x13:=X1/X3+X3/X1;
x23:=X2/X3+X3/X2;
hX:=Matrix(F,3,3,[2,x01,x02, x01,2,x12, x03,x13,x23]);
Determinant(hX) eq 0;
\end{verbatim}

\subsection*{Proof of Lemma \ref{lem:3}}
\begin{verbatim}
R<w0,w1,a01,a0j,a1j>:=PolynomialRing(Rationals(),5);
F:=FieldOfFractions(R);
I:=ideal<R|w0^2+w1^2-w0*w1*a01,a01^2+a0j^2+a1j^2-a01*a0j*a1j-4>;
x:=(w1^2-w0^2)/(a1j*w1-a0j*w0);
num:=Numerator(x^2-w0*a0j*x+w0^2);
num in I;
\end{verbatim}
\begin{verbatim}
R<w0,w1,wi,wj,aij>:=PolynomialRing(Rationals(),5);
M:=Matrix(R,3,3,[
2*w0^2,w0^2+w1^2,w0^2+wi^2,
w0^2+w1^2,2*w1^2,w1^2+wi^2,
w0^2+wj^2,w1^2+wj^2,wi*wj*aij]);
Determinant(M) eq -(w1^2-w0^2)^2*(wi*wj*aij-wi^2-wj^2);
\end{verbatim}

\subsection*{Proof of Lemma \ref{lem:w1w2+w3}}
\begin{verbatim}
F<X1,X2,X3>:=FunctionField(Rationals(),3);
x01:=1/X1+X1;
x02:=1/X2+X2;
x03:=1/X3+X3;
x12:=X1/X2+X2/X1;
x13:=X1/X3+X3/X1;
x23:=X2/X3+X3/X2;
(X1*X2*X3+1)*(x01*x02+x03-x12) eq
(X1*X2+X3)*(x01*x02*x03+2-1/2*(x12*x03+x13*x02+x23*x01));
\end{verbatim}

\subsection*{Proof of Theorem \ref{thm:main}}
\begin{verbatim}
d:=3;
d2s:=&cat[[[i,j]:j in [i+1..d]]:i in [0..d-1]];
d2:=[Seqset(s):s in d2s];
R:=PolynomialRing(Rationals(),#d2+2);
q:=R.(#d2+2);
r:=R.(#d2+1);
X:=func<i,j|R.Position(d2,{i,j})>;
g:=func<i,j,k|X(i,j)^2+X(i,k)^2+X(j,k)^2-X(i,j)*X(i,k)*X(j,k)-4>;
h:=func<i,j,k,l|(X(k,l)^2-4)*X(i,j)
 -X(k,l)*(X(k,i)*X(l,j)+X(k,j)*X(l,i))
 +2*(X(k,i)*X(k,j)+X(l,i)*X(l,j))>;

eigenP:=Matrix(R,4,4,[
1,1/2*q^2-q,1/2*q^2,q-2,
1,1/2*q,-1/2*q,-1,
1,-1/2*q+1,-1/2*q,q-2,
1,-1/2*q,1/2*q,-1]);
P:=func<i,j|eigenP[i+1,j+1]>;
n:=&+[P(0,i):i in [0..d]];
n eq q^2-1;

e:=func<i|-n+&+[P(i,j)^2:j in [0..d]]
 +&+[P(i,j[1])*P(i,j[2])*X(j[1],j[2]):j in d2s]>;
s3:=[Setseq(x):x in Subsets({0..d},3)];
eq7:=[g(i[1],i[2],i[3]):i in s3] cat
 [h(0^i,1^i,2^i,3^i):i in Sym({0..d})] cat
 [e(i):i in [1..d]] cat
 [r^2-(17*q-1)*(q-1)];
I:=ideal<R|eq7>;

b1:=X(1,2)-2;
b2:=(q^2-4)*X(1,2)+2*q^2-12;
b3:=q^2*X(1,2)+2*q^2-4;
a12D:=q^2*(q+1)*(q-3);
a12Nplus:= 2*(-q^4+2*q^3+4*q^2-10*q+1 +(q-1)*r);
a12Nminus:=2*(-q^4+2*q^3+4*q^2-10*q+1 -(q-1)*r);
b4plus:=a12D*X(1,2)-a12Nplus;
b4minus:=a12D*X(1,2)-a12Nminus;
q*(q-1)*b1*b2*b3*b4plus*b4minus in I;

I1:=ideal<R|I,b1>;
c1:=X(0,1)+n-2;
a01D:=q^2-2*q-1;
a01N:=q^3-3*q^2-q+7;
c2:=a01D*X(0,1)-a01N;
X(0,3)+n-2 in I1;
(q-1)^2*c1*c2 in I1;
I11:=ideal<R|I1,c1>;
(q-1)*(q+1)*(q-2)*(q^2-5)*(X(1,3)-2) in I11;

I12:=ideal<R|I1,c2>;
a13D:=q^2-2*q-1;
a13N:=-(q^3-q^2-q-3);
-a13N+2*a13D eq (q+1)*(q^2-5);
-a13N-2*a13D eq (q-1)^3;
(q-1)*(q+1)*(q-2)*(q^2-5)*(a13D*X(1,3)-a13N) in I12;
a01:=a01N/a01D;
a03:=-(n-2);
a13:=a13N/a13D;
FA<w3>:=PolynomialRing(FieldOfFractions(R));
FA3:=FA/ideal<FA|w3^2+1-a03*w3>;
QA3:=FieldOfFractions(FA3);
w1:=QA3!((w3^2-1)/(a13*w3-a01));
w1 eq QA3!((-(q-3)*w3+(q-1))/(q^2-2*q-1));

I2:=ideal<R|I,b2>;
q*(q-1)*(q+1)*(q-3)*(n-4)*(X(0,2)+2) in I2;
q*(q-1)*(q+1)*(q-3)*(n-4)*(X(1,3)-2) in I2;
q*(q-1)*(q+1)*(q-3)*(n-4)*((n-3)*X(0,1)-2*(n-5)) in I2;

I3:=ideal<R|I,b3>;
c4:=q^2*X(0,2)+2*(n-1);
c5:=q*X(0,2)+2;
X(0,3)-2 in I3;
(q-1)^2*c4*c5 in I3;

I31:=ideal<R|I3,c4>;
(q-1)*(q+1)*(q-2)*(X(0,1)-2) in I31;

I32:=ideal<R|I3,c5>;
(q-1)*(q+1)*(q-2)*(q*X(0,1)+2) in I32;
q^2*X(1,2)+2*(n-1) in I32;
a02:=-2/q;
a01:=-2/q;
a12:=-2*(n-1)/q^2;
a01*a02-a12 eq 2;

I4:=ideal<R|I,b4plus>;
a01D:=2*q*(q+1);
a01N:=-(q-1)*(q-2)+(q+2)*r;
a01:=a01D*X(0,1)-a01N;
a02D:=2*q*(q-3);
a02N:=(q+2)*(q-1)-(q-2)*r;
a02:=a02D*X(0,2)-a02N;
a03D:=2*(q+1)*(q-3);
a03N:=5*q^2-2*q-19-(q-1)*r;
a03:=a03D*X(0,3)-a03N;
a12D:=q^2*(q+1)*(q-3);
a12N:=2*(-q^4+2*q^3+4*q^2-10*q+1+(q-1)*r);
a12:=a12D*X(1,2)-a12N;
a13D:=a02D;
a13N:=-a02N;
a13:=a13D*X(1,3)-a13N;
a23D:=a01D;
a23N:=-a01N;
a23:=a23D*X(2,3)-a23N;
q*(q+1)*(q-1)^2*(q-3)^2*(q^2-5)*a01 in I4;
q*(q+1)^2*(q-1)^2*(q-3)*(q^2-5)*a02 in I4;
a03 in I4;
a12 in I4;
q*(q+1)^2*(q-1)^2*(q-3)*(q^2-5)*a13 in I4;
q*(q+1)*(q-1)^2*(q-3)^2*(q^2-5)*a23 in I4;
b4pp:=Evaluate(b4plus,[0,0,0,2,0,0,r,q]);
(b4pp+2*(q-1)*r)^2-4*(q-1)^2*(17*q-1)*(q-1) eq
 16*q^2*(q+1)*(q-1)^2*(q-3)*(q^2-5);
b4pm:=Evaluate(b4plus,[0,0,0,-2,0,0,r,q]);
(b4pm+2*(q-1)*r)^2-4*(q-1)^2*(17*q-1)*(q-1) eq
 -64*q^2*(q+1)*(q-3);
q^2*(q+1)*(q-1)^2*(q-3)*(q^2-5)*(X(0,1)*X(0,2)+X(0,3)-X(1,2)) in I4;
I41:=ideal<R|I4,X(0,1)*X(0,2)*X(0,3)+2
 -1/2*(X(1,2)*X(0,3)+X(1,3)*X(0,2)+X(2,3)*X(0,1))>;
q*(q-1)^2*(q^2-5)^2 in I41;
\end{verbatim}

Conversely,
\begin{verbatim}
Rqr<rr,qq>:=PolynomialRing(Rationals(),2);
F:=FieldOfFractions(Rqr/ideal<Rqr|rr^2-(17*qq-1)*(qq-1)>);
d:=3;
d2s:=&cat[[[i,j]:j in [i+1..d]]:i in [0..d-1]];
d2:=[Seqset(s):s in d2s];
q:=F!qq;
r:=F!rr;
R:=PolynomialRing(F,#d2);
X:=func<i,j|R.Position(d2,{i,j})>;
g:=func<i,j,k|X(i,j)^2+X(i,k)^2+X(j,k)^2-X(i,j)*X(i,k)*X(j,k)-4>;
h:=func<i,j,k,l|(X(k,l)^2-4)*X(i,j)
 -X(k,l)*(X(k,i)*X(l,j)+X(k,j)*X(l,i))
 +2*(X(k,i)*X(k,j)+X(l,i)*X(l,j))>;
eigenP:=Matrix(F,4,4,[
 1,1/2*q^2-q,1/2*q^2,q-2,
 1,1/2*q,-1/2*q,-1,
 1,-1/2*q+1,-1/2*q,q-2,
 1,-1/2*q,1/2*q,-1]);
P:=func<i,j|eigenP[i+1,j+1]>;
n:=&+[P(0,i):i in [0..d]];
n eq q^2-1;
e:=func<i|-n+&+[P(i,j)^2:j in [0..d]]
 +&+[P(i,j[1])*P(i,j[2])*X(j[1],j[2]):j in d2s]>;
s3:=[Setseq(x):x in Subsets({0..d},3)];
eq7rq:=[g(i[1],i[2],i[3]):i in s3] cat
 [h(0^i,1^i,2^i,3^i):i in Sym({0..d})] cat
 [e(i):i in [1..d]];
\end{verbatim}

\begin{verbatim}
subs1:=[-n+2,-n+2,-n+2,2,2,2];
&and[Evaluate(f,subs1) eq 0:f in eq7rq];
\end{verbatim}

\begin{verbatim}
a:=-(q-3)/(q^2-2*q-1);
b:=(q-1)/(q^2-2*q-1);
a^2+a*b*(-(n-2))+b^2 eq 1;
a01:=(q^3-3*q^2-q+7)/(q^2-2*q-1);
a*(-(n-2))+2*b eq a01;
a13:=(-q^3+q^2+q+3)/(q^2-2*q-1);
2*a-b*(n-2) eq a13;
subs2:=[a01,a01,-(n-2),2,a13,a13];
&and[Evaluate(f,subs2) eq 0:f in eq7rq];
\end{verbatim}

\begin{verbatim}
a01_iii:=2*(n-5)/(n-3);
subs3:=[a01_iii,-2,a01_iii,-a01_iii,2,-a01_iii];
&and[Evaluate(f,subs3) eq 0:f in eq7rq];
\end{verbatim}

\begin{verbatim}
a02_iv:=-2*(n-1)/(n+1);
subs4:=[2,a02_iv,2,a02_iv,2,a02_iv];
&and[Evaluate(f,subs4) eq 0:f in eq7rq];
\end{verbatim}

\begin{verbatim}
subs5:=[-2/q,-2/q,2,-2*(n-1)/(n+1),-2/q,-2/q];
&and[Evaluate(f,subs5) eq 0:f in eq7rq];
\end{verbatim}

\begin{verbatim}
a01D:=2*q*(q+1);
a01N:=-(q-1)*(q-2)+(q+2)*r;
a01:=a01N/a01D;
a02D:=2*q*(q-3);
a02N:=(q+2)*(q-1)-(q-2)*r;
a02:=a02N/a02D;
a03D:=2*(q+1)*(q-3);
a03N:=5*q^2-2*q-19-(q-1)*r;
a03:=a03N/a03D;
a12D:=q^2*(q+1)*(q-3);
a12N:=2*(-q^4+2*q^3+4*q^2-10*q+1+(q-1)*r);
a12:=a12N/a12D;
a13D:=a02D;
a13N:=-a02N;
a13:=a13N/a13D;
a23D:=a01D;
a23N:=-a01N;
a23:=a23N/a23D;
subs6:=[a01,a02,a03,a12,a13,a23];
&and[Evaluate(f,subs6) eq 0:f in eq7rq];
\end{verbatim}

\subsection*{Proof of Corollary \ref{cor}}
\begin{verbatim}
a01-2 eq -8*q*(q+1)*(q-3)^2/(2*q*(q+1)*(5*q^2+q+2+(q+2)*r));
a01+2 eq -8*q*(q+1)*(q^2-5)/(2*q*(q+1)*(3*q^2+7*q-2-(q+2)*r));
\end{verbatim}

\subsection*{Isolation}
\begin{verbatim}
n:=15;
A0:=ScalarMatrix(n,1);
J:=Parent(A0)![1:i in [1..n^2]];
LO3:=LineGraph(OddGraph(3));
A1:=AdjacencyMatrix(LO3);
A2:=A1^2-A1-4*A0;
A3:=J-A0-A1-A2;
DM:=DistanceMatrix(LO3);
DM eq A1+2*A2+3*A3;

hermitianConjugate:=
 func<H|Parent(H)![ComplexConjugate(x):x in Eltseq(Transpose(H))]>;

complexHadamard:=function(xyz)
 AA:=[ChangeRing(A,Parent(xyz[1])):A in [A1,A2,A3]];
 return A0+xyz[1]*AA[1]+xyz[2]*AA[2]+xyz[3]*AA[3];
end function;

spanCondition:=function(H)
 F:=Parent(H[1,1]);
 MnF:=Parent(H);
 n:=Nrows(H);
 Es:=[MnF|0:i in [1..n]];
 for i in [1..n] do
  Es[i][i,i]:=1;
 end for;
 EsF:=[MnF|e:e in Es];
 Hs:=hermitianConjugate(H);
 Vn:=VectorSpace(F,n^2);
 bracket:=sub<Vn|[Vn|Eltseq(v*Hs*w*H-Hs*w*H*v):v,w in EsF]>;
 return Dimension(bracket) eq n^2-2*n+1;
end function;

F<s>:=QuadraticField(-15);
y:=(-7+s)/8;
H:=complexHadamard([1,y,1]);
H*hermitianConjugate(H) eq n*A0;
spanCondition(H);

F<s>:=QuadraticField(-11);
x:=(5+s)/6;
H:=complexHadamard([x,-1,x]);
H*hermitianConjugate(H) eq n*A0;
not spanCondition(H);

F<s>:=QuadraticField(-15);
x:=(-1+s)/4;
H:=complexHadamard([x,x^(-1),1]);
H*hermitianConjugate(H) eq n*A0;
not spanCondition(H);

F<s>:=QuadraticField(201);
Z:=(53-3*s)/10;
R<T>:=PolynomialRing(F);
K<z>:=ext<F|T^2-Z*T+1>;
z+1/z eq Z;
x:=1/144*((-5*Z+31)*z-25*Z+155);
xb:=1/144*((-5*Z+31)*z^(-1)-25*Z+155);
y:=1/144*((25*Z-155)*z+5*Z-31);
yb:=1/144*((25*Z-155)*z^(-1)+5*Z-31);
x*xb eq 1;
y*yb eq 1;
H:=complexHadamard([x,y,z]);
H*hermitianConjugate(H) eq n*A0;
spanCondition(H);
\end{verbatim}

\subsection*{Remark \ref{rem:1}}
\begin{verbatim}
mustbe0:=func<a01v|(2*q*(q+1)*a01v+(q-1)*(q-2))^2-((q+2)*r)^2>;
mustbe0(0) eq -8*q*(q+1)*(q-1)*(2*q+7);
mustbe0(1) eq -8*q*(q+1)*(q^2+6*q-8);
mustbe0(-1) eq -8*q*(q+1)*(2*q^2+3*q-6);
mustbe0(2) eq 8*q*(q+1)*(q-3)^2;
mustbe0(-2) eq -8*q*(q+1)*(q^2-5);
\end{verbatim}

\subsection*{Table \ref{tab:1}}
\begin{verbatim}
HWminus1:=function(w)
 I3:={1..3};
 H3q:={w[i1]*w[i2]/w[i3]:i1,i2,i3 in I3
   |#[i:i in [i1,i2,i3]|i eq 3] ne 2};
 H3q4:={w[i1]*w[i2]/w[i3]:i1,i2,i3 in I3
   |#[i:i in [i1,i2,i3]|i eq 3] ne 2 and {i1,i2,i3} ne {1,3}};
 plus:=[{w[i]^2:i in I3} join H:H in [H3q,H3q4]];
 return {(p join {x^(-1):x in p} join 
  {w[i1]*w[i2]/(w[j1]*w[j2]):i1,i2,j1,j2 in I3})
  diff {1}:p in plus};
end function;
Rw<w1,w2,w3>:=FunctionField(Rationals(),3);
HWminus1([w1,w1,w1]) eq 
 {&join{{w1^s,w1^(s*2)}:s in {1,-1}}};
HWminus1([w1,w1,w3]) eq 
 {&join{{w^s,w^(s*2)}:s in {1,-1},w in {w1,w3}} join
 &join{{(w1^2/w3)^s,(w3/w1)^s,(w3/w1)^(s*2)}:s in {1,-1}}};
HWminus1([w1,-1,w1]) eq {{-1} join 
 &join{{s1*w1^s,s1*w1^(s*2)}:s,s1 in {1,-1}}};
HWminus1([1,w2,1]) eq {&join{{w2^s,w2^(s*2)}:s in {1,-1}}};
HWminus1([w1,w1^(-1),1]) eq {{w1^(s*k):s in {1,-1},k in {1..4}}};
HWminus1([w1,w2,-w1*w2]) eq {{-1} join
 {s0*w^(s*k):w in {w1,w2},s,s0 in {1,-1},k in {1,2}} join
 &join{{s0*w1^s1*w2^s2,(w1^s1*w2^s2)^2}:s0,s1,s2 in {1,-1}}
 join &join{{s0*(w1^2*w2^(-1))^s,s0*(w1^(-1)*w2^2)^s}
  :s,s0 in {1,-1}}};
// (i)
Rq<q>:=FunctionField(Rationals());
(-q^2+3)^2-2 eq q^4-6*q^2+7;
// (ii)
Rw3<w3>:=FunctionField(Rq);
A:=-(q-3)/(q^2-2*q-1);
B:=(q-1)/(q^2-2*q-1);
(A^2+B^2-1)/(A*B) eq q^2-3;
w1:=A*w3+B;
(A/w3+B)-1/w1 eq 1/w1*A*B*(w3+1/w3+(q^2-3));
// (iii)
(2*(q^2-6)/(q^2-4))^2-2 eq 2*(q^4-16*q^2+56)/(q^2-4)^2;
// (iv)
(-2*(q^2-2)/q^2)^2-2 eq 2*(q^4-8*q^2+8)/q^4;
\end{verbatim}

\subsection*{Proof of Proposition \ref{prop:1}}
(iii)$\not\cong$(vi)
\begin{verbatim}
Rq<q>:=FunctionField(Rationals());
k3a:=(q^2-6)/(q^2-4);
k3b:=2*(q^4-16*q^2+56)/(q^2-4)^2;
ra1:=(k3a+(q-1)*(q-2)/(2*q*(q+1)))/(q+2);
fac:=Factorization(Numerator(ra1^2-(17*q-1)*(q-1)));
#fac eq 1 and Degree(fac[1][1]) gt 1;
ra2:=(-k3a+(q-1)*(q-2)/(2*q*(q+1)))/(q+2);
fac:=Factorization(Numerator(ra2^2-(17*q-1)*(q-1)));
#fac eq 1 and Degree(fac[1][1]) gt 1;
rb1:=(k3b+(q-1)*(q-2)/(2*q*(q+1)))/(q+2);
fac:=Factorization(Numerator(rb1^2-(17*q-1)*(q-1)));
#fac eq 1 and Degree(fac[1][1]) gt 1;
rb2:=(-k3b+(q-1)*(q-2)/(2*q*(q+1)))/(q+2);
fac:=Factorization(Numerator(rb2^2-(17*q-1)*(q-1)));
#fac eq 1 and Degree(fac[1][1]) gt 1;
\end{verbatim}

(i)$\not\cong$(ii)
\begin{verbatim}
n:=q^2-1;
Numerator((2-n)-(q^3-3*q^2-q+7)/(q^2-2*q-1))
 eq -(q-2)*(q+1)*(q^2-5);
Numerator((n^2-4*n+2)-(q^3-3*q^2-q+7)/(q^2-2*q-1))
 eq (q-2)*(q+1)^2*(q^3-2*q^2-4*q+7);
\end{verbatim}

(iv)$\not\cong$(v)
\begin{verbatim}
Numerator(-2*(n-1)/(n+1)-(-2)/q)
 eq -2*(q-2)*(q+1);
Numerator(2*(n^2-6*n+1)/(n+1)^2-(-2)/q)
 eq 2*(q-2)*(q+1)*(q^2+2*q-4);
\end{verbatim}

\subsection*{Proof of Proposition \ref{prop:2}}

\begin{verbatim}
ra1:=(2+(q-1)*(q-2)/(2*q*(q+1)))/(q+2);
fac:=Factorization(Numerator(ra1^2-(17*q-1)*(q-1)));
#fac eq 1 and Degree(fac[1][1]) gt 1;
ra2:=(-2+(q-1)*(q-2)/(2*q*(q+1)))/(q+2);
fac:=Factorization(Numerator(ra2^2-(17*q-1)*(q-1)));
#fac eq 1 and Degree(fac[1][1]) gt 1;
\end{verbatim}

\subsection*{Proof of Lemma \ref{lem:NS}}
\begin{verbatim}
R<r,q>:=PolynomialRing(Rationals(),2);
F:=FieldOfFractions(R);
n:=q^2-1;
fr:=r^2-(17*q-1)*(q-1);
B1:=Matrix(R,4,4,[
0,1,0,0,
q^2/2-q,(q-2)^2/4,(q-2)^2/4,q*(q-4)/4,
0,q*(q-2)/4,q*(q-2)/4,q^2/4,
0,(q-4)/2,(q-2)/2,0]);

B2:=Matrix(R,4,4,[
0,0,1,0,
0,q*(q-2)/4,q*(q-2)/4,q^2/4,
q^2/2,q^2/4,q^2/4,q^2/4,
0,q/2,(q-2)/2,0]);

B3:=Matrix(R,4,4,[
0,0,0,1,
0,(q-4)/2,(q-2)/2,0,
0,q/2,(q-2)/2,0,
q-2,0,0,q-3]);

B0:=ScalarMatrix(4,R!1);
BB:=[B0,B1,B2,B3];
pijk:=func<i,j,k|BB[i+1][j+1,k+1]>;

isSymNbas:=function(ajk)
 aijs:=[[ajk[1],ajk[2],ajk[3]],
  [1,ajk[4],ajk[5]],[1,1,ajk[6]]];
 aij:=func<i,j|aijs[i+1,j]>;
 ff:=[F|&+[pijk(j,k,i)*(aij(j,k)^2-2):j,k in [0..3]|j lt k]
  +&+[pijk(j,j,i):j in [0..3]]:i in [1..3]];
 II:=[ideal<R|fr,Numerator(ff[k])>:k in [1..3]];
 return [Basis(EliminationIdeal(II[k],{q})):k in [1..3]];
end function;

// case (i)
subs1:=[-n+2,-n+2,-n+2,2,2,2];
bass:=isSymNbas(subs1);
[Factorization(bass[i][1]):i in [1..3]] eq 
 [[<q-2,1>,<q-1,1>,<q+1,1>,<q+2,1>]:i in [1..3]];

// case (ii)
a01:=(q^3-3*q^2-q+7)/(q^2-2*q-1);
a13:=(-q^3+q^2+q+3)/(q^2-2*q-1);
subs2:=[a01,a01,-(n-2),2,a13,a13];
bass:=isSymNbas(subs2);
b1:=q^4-10*q^2+4*q+17;
fac:= [[<q-2,1>,<q-1,1>,<q+1,1>,<b1,1>],
 [<q-2,1>,<q-1,1>,<q+1,1>,<b1,1>],
 [<q-2,1>,<q-1,1>,<q+1,1>,<q+2,1>]];
[Factorization(bass[i][1]):i in [1..3]] eq fac;

// case (iii)
a01_iii:=2*(n-5)/(n-3);
subs3:=[a01_iii,-2,a01_iii,-a01_iii,2,-a01_iii];
bass:=isSymNbas(subs3);
b1:=q^6-13*q^4+28*q^2+64;
b2:=q^4-9*q^2+24;
[Factorization(bass[i][1]):i in [1..3]] eq
 [[<b1,1>],[<b2,1>],[<b1,1>]];

// case (iv)
a02_iv:=-2*(n-1)/(n+1);
subs4:=[2,a02_iv,2,a02_iv,2,a02_iv];
bass:=isSymNbas(subs4);
bass eq [[(q-2)*(q-1)*(q+1)*(q+2)]:i in [1..3]];

// case (v)
subs5:=[-2/q,-2/q,2,-2*(n-1)/(n+1),-2/q,-2/q];
bass:=isSymNbas(subs5);
b1:=(q-2)*(q-1)*(q+1)*(q^2-2*q-4);
b3:=(q-2)*(q-1)*(q+1)*(q+2);
bass eq [[b1],[b1],[b3]];

// case (vi)
a01D:=2*q*(q+1);
a01N:=-(q-1)*(q-2)+(q+2)*r;
a01:=a01N/a01D;
a02D:=2*q*(q-3);
a02N:=(q+2)*(q-1)-(q-2)*r;
a02:=a02N/a02D;
a03D:=2*(q+1)*(q-3);
a03N:=5*q^2-2*q-19-(q-1)*r;
a03:=a03N/a03D;
a12D:=q^2*(q+1)*(q-3);
a12N:=2*(-q^4+2*q^3+4*q^2-10*q+1+(q-1)*r);
a12:=a12N/a12D;
a13:=-a02;
a23:=-a01;
bass:=isSymNbas([a01,a02,a03,a12,a13,a23]);
bass eq [
 [q^3*(q-3)^2*(q-1)*(q+1)^2*
  (q^9-q^8-12*q^7+14*q^6+49*q^5
  +51*q^4-894*q^3-464*q^2+4664*q-272)],
 [q^3*(q-3)^2*(q-2)*(q-1)*(q+1)^2*
  (q^7+3*q^6-4*q^5+2*q^4+57*q^3-q^2-86*q+92)],
 [q^2*(q-3)^2*(q-1)*(q+1)^2*
  (q^8-2*q^7+66*q^5-273*q^4-288*q^3+1344*q^2-288*q+16)]];
\end{verbatim}

\subsection*{Proof of Proposition \ref{prop:N2}}
\begin{verbatim}
R<c111,c112,c121,c122,c211,c212,c221,c222,
 w1,w2case4,w3case2,r,q>
 :=PolynomialRing(Rationals(),13);
varname:=[[i,j,k]:i,j,k in [1,2]];
c:=func<i,j,k|R.Position(varname,[i,j,k])>;
fr:=r^2-(17*q-1)*(q-1);
w0:=1;

B1:=Matrix(R,4,4,[
0,1,0,0,
q^2/2-q,(q-2)^2/4,(q-2)^2/4,q*(q-4)/4,
0,q*(q-2)/4,q*(q-2)/4,q^2/4,
0,(q-4)/2,(q-2)/2,0]);

B2:=Matrix(R,4,4,[
0,0,1,0,
0,q*(q-2)/4,q*(q-2)/4,q^2/4,
q^2/2,q^2/4,q^2/4,q^2/4,
0,q/2,(q-2)/2,0]);

B3:=Matrix(R,4,4,[
0,0,0,1,
0,(q-4)/2,(q-2)/2,0,
0,q/2,(q-2)/2,0,
q-2,0,0,q-3]);

B0:=ScalarMatrix(4,R!1);
BB:=[B0,B1,B2,B3];
pijk:=func<i,j,k|BB[i+1][j+1,k+1]>;

cijk:=function(i,j,k)
 if 0 in {i,j,k} then
  if [i,j,k] in {[0,3,3],[3,0,3],[3,3,0]} then
   return 1;
  else 
   return 0;
  end if;
 elif 3 in {i,j,k} then
  if {3} eq {i,j,k} then
   return pijk(3,3,3)-1;
  else
   return 0;
  end if;
 else  
  return c(i,j,k);
 end if;
end function; 

fx:=[cijk(1,j,k)+cijk(2,j,k)-pijk(j,k,3):j,k in [1,2]];
fy:=[cijk(j,1,k)+cijk(j,2,k)-pijk(j,k,3):j,k in [1,2]];
fz:=[cijk(j,k,1)+cijk(j,k,2)-pijk(j,k,3):j,k in [1,2]];
fxyz:=fx cat fy cat fz;

test:=function(fa,ww)
 alpha:=func<i|ww[i+1]>;
 ff:=&+[cijk(i,j,k)*alpha(i)^2/(alpha(j)*alpha(k))
 :i,j,k in [0..3]];
 gg:=&+[cijk(i,j,k)*alpha(j)*alpha(k)/alpha(i)^2
 :i,j,k in [0..3]];
 I:=ideal<R|[fr,fa,Numerator(ff),Numerator(gg)] cat fxyz>;
 return Basis(EliminationIdeal(I,{q}));
end function;

test2:=function(fa,ww)
 alpha:=func<i|ww[i+1]>;
 t1:=&+[pijk(i,j,1)*pijk(3,k,i)*alpha(i)^2/(alpha(j)*alpha(k))
  :i,j,k in [0..3]]; 
 t2:=&+[pijk(i,j,2)*pijk(3,k,i)*alpha(i)^2/(alpha(j)*alpha(k))
  :i,j,k in [0..3]]; 
 I1:=ideal<R|[fr,fa,Numerator(t1)]>;
 I2:=ideal<R|[fr,fa,Numerator(t2)]>;
 return [Basis(EliminationIdeal(I1,{q})),
  Basis(EliminationIdeal(I1,{q}))];
end function;

// case (i)
fa1:=w1^2+(q^2-3)*w1+1;
ww:=[w0,w1,w1,w1];
test(fa1,ww) eq [(q+1)*(q-1)*(q^2-5)];
test2(fa1,ww) eq [[(q-2)*(q+1)*(q-1)*(q^2-5)]:i in [1,2]];

// case (ii)
fa3:=w3case2^2+(q^2-3)*w3case2+1;
w2:=(-(q-3)*w3case2+q-1)/(q^2-2*q-1);
ww:=[w0,w2,w2,w3case2];
test(fa3,ww) eq [(q+1)*(q-1)*(q^2-5)];
test2(fa3,ww) eq 
 [[(q-2)*(q+1)*(q-1)*(q^2-5)*(q^2-2*q-1)^2]:i in [1,2]];

// case (iii)
fa1:=R!((q^2-4)*(w1^2-2*(q^2-6)/(q^2-4)*w1+1));
ww:=[w0,w1,-1,w1];
test(fa1,ww) eq [(q^2+4)*(q^2-5)];
test2(fa1,ww) eq 
 [[(q-2)*(q+1)*(q^2-5)*(q^5-5*q^4+12*q^3-24*q^2+64)]:i in [1,2]];

// case (iv)
fa2:=R!(q^2*(w2case4^2+2*(q^2-2)/q^2*w2case4+1));
ww:=[w0,1,w2case4,1];
test(fa2,ww) eq [q^2*(q+1)*(q-1)];
test2(fa2,ww) eq [[q^2*(q-2)*(q+1)*(q-1)]:i in [1,2]];

// case (v)
fa1:=R!(q*(w1^2+2/q*w1+1));
ww:=[w0,w1,1/w1,1];
test(fa1,ww) eq [q^2*(q+1)*(q-1)];
test2(fa1,ww) eq [[q*(q-2)*(q+1)*(q-1)]:i in [1,2]];

// case (vi)
a01D:=2*q*(q+1);
a01N:=-(q-1)*(q-2)+(q+2)*r;
a01:=a01N/a01D;
a02D:=2*q*(q-3);
a02N:=(q+2)*(q-1)-(q-2)*r;
a02:=a02N/a02D;
a03D:=2*(q+1)*(q-3);
a03N:=5*q^2-2*q-19-(q-1)*r;
a03:=a03N/a03D;
a12D:=q^2*(q+1)*(q-3);
a12N:=2*(-q^4+2*q^3+4*q^2-10*q+1+(q-1)*r);
a12:=a12N/a12D;
a13:=-a02;
a23:=-a01;
fa1:=a01D*w1^2-a01N*w1+a01D;
w2:=(a01*w1-2)/(a12*w1-a02);
w3:=(a01*w1-2)/(a13*w1-a03);
ww:=[w0,w1,w2,w3];
test(fa1,ww) eq [q^12*(q-3)^12*(q-1)*(q+1)^6*(q-1/9)*(q^2-5)^3];
test2(fa1,ww) eq [[q^9*(q-3)^12*(q-1)^2*(q+1)^7*(q^2-5)^4*
 (q^4+6/19*q^3-48/19*q^2+8/19*q+16/19)]:i in [1,2]];
\end{verbatim}

\newpage
\section{Detailed proof of Remark~\ref{rem:2}}

\begin{lem}\label{lem:P1}
Let $d$ be a square-free positive integer, and let $u,v$ be positive integers
satisfying $u^2-dv^2=1$ and $u>1$. Let $a$ be a positive integer.
If $x$ and
$y$ are positive integers satisfying $x^2-dy^2=a$ and $x>\sqrt{a(u+1)/2}$, 
then there exist
integers $x_0,y_0,n$ with $0<x_0\leq\sqrt{a(u+1)/2}$ such that
\begin{equation}\label{Pe1}
x+\sqrt{d}y=(u+\sqrt{d}v)^n(x_0+\sqrt{d}y_0).
\end{equation}
\end{lem}
\begin{proof}
Suppose that the statement is false, and choose the minimal
counterexample $x>\sqrt{a(u+1)/2}$ with $x^2-dy^2=a$, not 
expressible in the form (\ref{Pe1}). Then
$2x^2>a(u+1)$, so $\frac{a}{x^2}<\frac{2}{u+1}$.
Thus
\begin{align*}
u^2x^2&>(u^2-1)(x^2-a)=d^2v^2y^2
\nexteq
(u-1)^2x^2\frac{u+1}{u-1}(1-\frac{a}{x^2})
\\&>
(u-1)^2x^2\frac{u+1}{u-1}(1-\frac{2}{u+1})
\nexteq
(u-1)^2x^2.
\end{align*}
This implies $ux>dvy>(u-1)x$, or equivalently,
\[
x>ux-dvy>0.
\]
Set
\begin{align*}
x_1&=ux-dvy,\\
y_1&=|uy-vx|.
\end{align*}
Then $x>x_1>0$ and $y_1\geq0$.
\begin{align}
x_1\pm\sqrt{d}y_1&=
ux-dvy\pm\sqrt{d}(uy-vx)
\nnexteq
(u\mp\sqrt{d}v)(x\pm\sqrt{d}y),
\label{Pe3}
\end{align}
and so
\begin{align*}
x_1^2-dy_1^2&=
(u^2-dv^2)(x-dy^2)
\nexteq
a.
\end{align*}
Note that if $y_1=0$, then $x_1=\sqrt{a}$, so
\begin{equation}\label{Pe5}
x_1\leq\sqrt{a(u+1)/2}.
\end{equation}
Since $x$ is the minimal counterexample, we have either
(\ref{Pe5}), or there exist
integers $x_0,y_0,n$ with $0<x_0\leq\sqrt{a(u+1)/2}$ such that
\begin{equation}\label{Pe2}
x_1+\sqrt{d}y_1=(u+\sqrt{d}v)^n(x_0+\sqrt{d}y_0).
\end{equation}

In the former case, (\ref{Pe3}) implies
\[
x+\sqrt{d}y=(u+\sqrt{d}v)(x_1+\sqrt{d}y_1),
\]
so (\ref{Pe1}) holds by setting $(x_0,y_0,n)=(x_1,y_1,1)$.

In the latter case, (\ref{Pe3}) and (\ref{Pe2}) imply
\[
x+\sqrt{d}y=(u+\sqrt{d}v)^{n+1}(x_0+\sqrt{d}y_0),
\]
so (\ref{Pe1}) again holds. 

In either case, we obtain a contradiction to the
the assumption that $x$ is a counterexample.
\end{proof}

\begin{exa}\label{exa:P}
Since the Pell equation 
\[
u^2-17v^2=1
\]
has a solution $(u,v)=(33,8)$,
Lemma~\ref{lem:P1} implies that all solutions of the equation
\begin{equation}\label{Pe4}
x^2-17y^2=64
\end{equation}
can be expressed by those with $x<\sqrt{64(33+1)/2}=8\sqrt{17}$
and powers of $33+8\sqrt{17}$. Observe that the non-negative solutions
$(x,y)$ of (\ref{Pe4}) with $x<8\sqrt{17}$ are
\[
\{(8,0),(9,1),(26,6)\}.
\]
Now Lemma~\ref{lem:P1} implies
\begin{align*}
&\{x+\sqrt{17}y\mid x,y\in\Z,\;x,y>0,\;x^2-17y^2=64\}
\\&\subset\{(33+8\sqrt{17})^n\cdot8\mid n\in\Z\}
\cup
\{(33+8\sqrt{17})^n(9+\sqrt{17})\mid n\in\Z\}
\\&\quad\cup
\{(33+8\sqrt{17})^n(26+6\sqrt{17})\mid n\in\Z\}
\end{align*}
\begin{verbatim}
33^2-17*8^2 eq 1;
[v:v in [1..8]|IsSquare(17*v^2+1)] eq [8];
[x:x in [1..Floor(8*Sqrt(17))]|IsSquare((x^2-64)/17)] 
 eq [8,9,26];
\end{verbatim}
\end{exa}

\begin{lem}\label{lem:P2}
Let $q$ be an even positive integer with $q\geq4$. Then
$\sqrt{(q-1)(17q-1)}$ is an integer if and only if
\begin{align}
q&\in
\{\frac{1}{34}(\tr((2177+528\sqrt{17})^{n}(433+105\sqrt{17}))+18)\mid n\in\Z\}
\label{Pe6}\\&=\{
\ldots,41210,10,26,110890,482812730,\ldots\}.
\notag
\end{align}
\end{lem}
\begin{proof}
Setting $x=17q-9$ and $y=\sqrt{(q-1)(17q-1)}$, we have
\begin{align*}
x^2-17y^2&=
(17q-9)^2-17r^2
\nexteq
17^2q^2-2\cdot 17\cdot 9q+81-17(q-1)(17q-1)
\nexteq
64.
\end{align*}
Thus, by Example~\ref{exa:P}, we have
\begin{align*}
x+\sqrt{17}y&\in
\{(33+8\sqrt{17})^n\cdot8\mid n\in\Z\}
\\&\quad\cup
\{(33+8\sqrt{17})^n(9+\sqrt{17})\mid n\in\Z\}
\\&\quad\cup
\{(33+8\sqrt{17})^n(26+6\sqrt{17})\mid n\in\Z\}
\end{align*}
Since $q$ is even, we have $x\equiv-9\pmod{34}$. 
On the other hand, we have
\begin{align*}
(33+8\sqrt{17})^n\cdot8=a+b\sqrt{17}&\implies
a\equiv8(-1)^n\pmod{34},\\
(33+8\sqrt{17})^n(9+\sqrt{17})=a+b\sqrt{17}&\implies
a\equiv9(-1)^{n}\pmod{34},\\
(33+8\sqrt{17})^n(26+6\sqrt{17})=a+b\sqrt{17}&\implies
a\equiv8(-1)^{n+1}\pmod{34},
\end{align*}
Thus
\begin{align}
x+\sqrt{17}y&\in
\{(33+8\sqrt{17})^{2n+1}(9+\sqrt{17})\mid n\in\Z\}
\notag\nexteq
\{(2177+528\sqrt{17})^{n}(433+105\sqrt{17})\mid n\in\Z\}
\label{Pe8}
\end{align}
Thus
\begin{equation}\label{Pe7}
x\in
\{\frac12\tr((2177+528\sqrt{17})^{n}(433+105\sqrt{17}))\mid n\in\Z\}
\end{equation}
Since $q=(x+9)/17$, we obtain (\ref{Pe6}).

Conversely, if (\ref{Pe6}) holds, then (\ref{Pe7}) holds, and hence
(\ref{Pe8}) holds for some integer $y$. Moreover, (\ref{Pe8}) implies
$x^2-17y^2=64$, and hence $y^2=(17q-1)(q-1)$.
\end{proof}

\begin{verbatim}
K<s>:=QuadraticField(17);
(33+8*s)^2 eq 2177+528*s;
(33+8*s)*(9+s) eq 433+105*s;
[1/34*(Trace((2177+528*s)^n*(433+105*s))+18):n in [-2..2]]
 eq [41210,10,26,110890,482812730];
\end{verbatim}

When (\ref{Pe6}) holds, $a_{0,1}$ is rational. This means that
$w_1$ is at most quadratic. The following computation suggests
that $w_1$ is always quadratic.
\begin{verbatim}
K<s>:=QuadraticField(17);
qs:=[1/34*(Trace((2177+528*s)^n*(433+105*s))+18):n in [-1000..1000]];
w1rational:=function(q)
 t,r:=IsSquare((17*q-1)*(q-1));
 if t then
  a01:=(-(q-1)*(q-2)+(q+2)*r)/(2*q*(q+1));
  return IsSquare(a01^2-4);
 else 
  return false;
 end if;
end function;
&and[not D(q):q in qs];
\end{verbatim}

\end{document}